\documentclass[10pt]{amsart}
\usepackage{amsfonts}
\usepackage[leqno]{amsmath}
\usepackage{amssymb}
\usepackage{latexsym}
\usepackage{epsfig}
\usepackage{subcaption}
\usepackage{graphicx,color,graphics}

\usepackage[T1]{fontenc}
\usepackage[utf8]{inputenc}
\usepackage{combelow}

\usepackage{newunicodechar}

\newtheorem{lemma}{Lemma}[section]

\newtheorem{theorem}[lemma]{Theorem}
\newtheorem{proposition}[lemma]{Proposition}

\newtheorem{corollary}[lemma]{Corollary}

\newcommand{\N}{\ifmmode{{\Bbb N}}\else{\mbox{${\Bbb N}$}}\fi}
\newcommand{\R}{\ifmmode{{\Bbb R}}\else{\mbox{${\Bbb R}$}}\fi}

\newcommand{\halfscript}[2]{%
  \mathord{\hbox{
    \valign{%
      \vfil##\vfil\cr
      \hbox{$#1$}\cr
      \hbox{$\scriptstyle#2$}\cr
    }%
    \kern\scriptspace
  }}%
}

\begin{document}
	
\title[Numerical stabilization for a mixture system with kind damping]{Numerical stabilization for a mixture system with kind damping} 

\author{K. Ammari}
\address{LR Analysis and Control of PDEs, LR 22ES03, Department of Mathematics, Faculty of Sciences of Monastir, University of Monastir, 5019 Monastir, Tunisia}
\email{kais.ammari@fsm.rnu.tn}

\author{V. Komornik}
\thanks{The author was supported by the following grants: NSFC No. 11871348, CAPES: No. 88881.520205/2020-01,  
MATH AMSUD: 21-MATH-03.}
\address{D\'epartement de math\'ematique, Universit\'e de Strasbourg, 7 rue René Descartes, 67084 Strasbourg Cedex, France}
\email{komornik@math.unistra.fr}

\author{M. Sep\'{u}lveda-Cortés}
\thanks{The author was supported by the following grants: Fondecyt-ANID project 1220869,
and ANID-Chile through Centro de Modelamiento  Matem\'atico  (FB210005)}
\address{DIM and CI$^2$MA, Universidad de Concepci\'on, Concepci\'on, Chile}
\email{maursepu@udec.cl}

\author{O. Vera-Villagrán}
\thanks{The author is partially financed by UTA MAYOR 2022-2023, 4764-22.}
\address{Departamento de Matem\'{a}tica, Universidad
de Tarapac\'{a}, Av. 18 de Septiembre 2222, Arica, 
Chile}
\email{opverav@academicos.uta.cl}

\date{}

\begin{abstract}
In this paper, we conduct a numerical analysis of the strong stabilization and polynomial decay of solutions for the initial boundary value problem associated with a system that models the dynamics of a mixture of two rigid solids with porosity. This mathematical model accounts for the complex interactions between the rigid components and their porous structure, providing valuable information on the mechanical behavior of such systems.
Our primary objective is to establish conditions under which stabilization is ensured and to rigorously quantify the rate of decay of the solutions. Using numerical simulations, we assess the effectiveness of different stabilization mechanisms and analyze the influence of key system parameters on the overall dynamics.  
\end{abstract}


\maketitle
\tableofcontents

\section{Introduction}
The theory of mixtures of solids has been extensively explored by various researchers over the decades. Foundational contributions include the works of Truesdell and Toupin (\cite{trues1}, 1960), Green and Naghdi (\cite{green1}, 1965; \cite{green2}, 1968), and Bowen and Wiese (\cite{bo1}, 1969; \cite{bo2}, 1976). These studies have laid the foundation for understanding complex interactions in solid mixtures.

\medskip 

Detailed presentations of these theories can be found in key articles such as those by Atkin and Craine (\cite{atkin}, 1976), Bedford and Drumheller (\cite{bel}, 1983) and Rajagopal and Tao (\cite{raja}, 1995). Bowen's book (\cite{bo2}, 1976) offers a comprehensive account of the theories related to mixtures of elastic solids. Meanwhile, Green and Steel (\cite{green3}, 1966) and Steel (\cite{steel}, 1967) formulated theories using spatial descriptions in which displacement gradients and relative velocity served as independent constitutive variables.

\medskip

The first theory utilizing a Lagrangian framework was introduced by Bedford and Stern (\cite{bed}, 1972). In their approach, the independent constitutive variables are the displacement gradients and the relative displacement. Over the years, significant interest has been directed toward examining the qualitative properties of this theory, contributing to its further refinement and application.

\medskip

In the following, we focus on a specific case within the linear theory of porous viscoelastic mixtures. Our analysis is restricted to the interaction between the porosity field,
$u= u({\pmb{\color{black}{x}}},\,t)$ and the temperature field, $v= v({\pmb{\color{black}{x}}},\,t)$, within a homogeneous and isotropic mixture. For this scenario, the governing equations for $u$ and $v$, assuming the absence of body loads, are represented by the following system of equations. We will now consider the following system:

\begin{eqnarray}
\left\lbrace
\label{101}
\begin{array}{l}
\rho_{1}u_{tt} - a_{11}\Delta u - a_{12}\Delta v - b_{11}\Delta u_{t} - b_{12}\Delta v_{t} + \alpha(u - v) = 0,  \\
\rho_{2}v_{tt} - a_{12}\Delta u - a_{22}\Delta v - b_{12}\Delta u_{t} - b_{22}\Delta v_{t} - \alpha(u - v) = 0 ,
\end{array}
\right. 
\end{eqnarray}
where we use bold letter for vector fields. The subscript $t$ represent the time derivative. We consider an isotropic elastic body $\Omega\subset \mathbb{R}^{3}$ of density $\rho_{1} > 0$ and $\rho_{2} > 0.$ Moreover, the matrix $A = (a_{ij})$ and $B = (b_{ij})$ are symmetric, positive definite and 
\begin{align}
\label{102}a_{11}>0, \qquad a_{11}a_{22} - a_{12}^{2} >0,
\end{align}
in orthogonal curvilinear 
coordinates, at point ${\pmb{\color{black}{x}}} = (x_{1},\,x_{2},\,x_{3}) \in \Omega$ and time $t > 0.$ We assume the domain $\Omega$ is a simply connected smooth subset of $\mathbb{R}^{3}$ with smooth boundary $\partial\Omega.$ The set of equations \eqref{101} is completed with homogeneous Dirichlet conditions:
\begin{align}
\label{103}u({\pmb{\color{black}{x}}},\,t) = v({\pmb{\color{black}{x}}},\,t)  = 0,\quad \forall\,({\pmb{\color{black}{x}}},\,t)\in \partial\Omega\times (0,\,+\infty).
\end{align}
These conditions correspond to a rigidly clamped structure. Initial condition are:
\begin{align}
\label{104}(u,\,v)({\pmb{\color{black}{x}}},\,0) = (u_{0}({\pmb{\color{black}{x}}}),\,v_{0}({\pmb{\color{black}{x}}})),\quad (u_{t},\,v_{t})({\pmb{\color{black}{x}}},\,0) = (u_{1}({\pmb{\color{black}{x}}}),\,v_{1}({\pmb{\color{black}{x}}})).
\end{align}
For the sake of simplicity, we consider a homogeneous material, so that all parameters $\rho,\,j$ are constants equal to $1.$ Following the idea given by \cite{kais} we study the asymptotic behavior for the system
\begin{eqnarray}
\left\lbrace
\label{105}
\begin{array}{l}
\rho_{1}u_{tt} - a_{11}\Delta u - a_{12}\Delta v + \alpha(u - v) + \partial_{t}^{\alpha,\,\eta}u = 0, \quad ({\pmb{\color{black}{x}}},\,t)\in \Omega\times (0,\,\infty),  \\
\rho_{2}v_{tt} - a_{12}\Delta u - a_{22}\Delta v - \alpha(u - v)  + \partial_{t}^{\alpha,\,\eta}v  = 0, \quad ({\pmb{\color{black}{x}}},\,t)\in \Omega\times (0,\,\infty),\\
u({\pmb{\color{black}{x}}},\,0) = u_{0}({\pmb{\color{black}{x}}}),\quad u_{t}({\pmb{\color{black}{x}}},\,0) = u_{1}({\pmb{\color{black}{x}}}),\quad  {\pmb{\color{black}{x}}}\in \Omega, \\
v({\pmb{\color{black}{x}}},\,0) = v_{0}({\pmb{\color{black}{x}}}),\quad v_{t}({\pmb{\color{black}{x}}},\,0) = v_{1}({\pmb{\color{black}{x}}}),\quad  {\pmb{\color{black}{x}}}\in \Omega,  \\
u({\pmb{\color{black}{x}}},\,t) = v({\pmb{\color{black}{x}}},\,t) = 0,\quad \forall\,({\pmb{\color{black}{x}}},\,t)\in\partial\Omega\times (0,\,+\infty),
\end{array}
\right. 
\end{eqnarray}
where the dissipative terms $- b_{11}\Delta u_{t} - b_{12}\Delta v_{t}$ and $- b_{12}\Delta u_{t} - b_{22}\Delta v_{t}$ are replaced by the following fractional derivatives kind damping $\partial_{t}^{\alpha,\,\eta}u$ and $\partial_{t}^{\alpha,\,\eta}v$ respectively. 
This new approximation can have many applications.
Researchers have developed models using fractional derivatives to describe the viscoelastic behavior of bituminous mixtures in pavement systems. These models enable the simulation of phenomena such as creep and recovery, providing deeper insight into the mechanical properties of both the mineral skeleton and the bituminous binder (see for example \cite{Lagos}).

\medskip

Our purpose in this work is to investigate the stability of the solutions of system \eqref{105}.

\medskip

The rest of the paper is divided into five sections. In Section \ref{sec2}, we show that system \eqref{105} can be replaced by an augmented system \eqref{209} obtained by coupling an equation with a suitable diffusion, and we study the functional energy associated to the system. In Section \ref{sec3}, we establish the existence and uniqueness of solutions of the system \eqref{209}; for this, we use \cite{kais, kaisI}. In Section \ref{sec4}, we prove the strong stability and the polynomial stability of the system \eqref{209}. In Sections \ref{sec5}, we study the polynomial stability. 

\medskip

Throughout this paper, $C$ is a generic constant, not necessarily the same on each occasion (it may change from line to line) and depending on the indicated quantities.

\section{Augmented Model} \label{sec2}
\label{augmodel}
Initially, we provide a brief review of fractional calculus. For
the fractional integral, there are several slightly different
definitions for the fractional derivative operator. We understood
the concept in the Caputo sense (see \cite{C1, C2, C3, KST}).
\\
Let $0 < \alpha < 1.$  The Caputo fractional integral of order
$\alpha$ is defined by
\begin{align}\nonumber
I^{\alpha}f(t) =
\frac{1}{\Gamma(\alpha)}\int_{0}^{t}(t - s)^{\alpha - 1}
f(s)ds,
\end{align}
where $\Gamma$ is the well-known gamma function, $f \in
L^{1}(0,\,+\infty).$ 
The Caputo fractional derivative operator of order $\alpha$ is
defined by
\begin{equation}
\nonumber 
D^{\alpha}f(t) = I^{1 - \alpha}Df(t):=
\frac{1}{\Gamma(1 - \alpha)}\int_{0}^{t}(t - s)^{-\alpha}
f'(s)ds,
\end{equation}
with $f \in W^{1,\,1}(0,\,+\infty).$ 
We note that Caputo definition of fractional derivative does possess a
very simple interpretation, that means, if the function $f(t)$
represents the strain history within a viscoelastic material whose
relaxation function is $[\Gamma(1 - \alpha)t^{\alpha}]^{-1}$ then
the material will experience at  any time $t$ a total stress given
the expression $D^{\alpha}f(t).$
Moreover, it easy to show that  $D^{\alpha}$ is a left inverse of
$I^{\alpha},$  but in general it is not a right inverse. Indeed, we have
\begin{eqnarray*}
D^{\alpha}I^{\alpha}f = f, \qquad I^{\alpha}D^{\alpha}f(t) =  f(t)
- f(0).
\end{eqnarray*}
For more properties of fractional
calculus  see \cite{SKM}. 

\medskip

In this paper, we consider different versions that
\eqref{201} and \eqref{202}.  In fact, Choi and
MacCamy \cite{CM} establish the following definition of fractional
integro-differential operators with exponential weight. Let  $0 <
\alpha < 1\,$ and $\eta \ge 0.$ The exponential fractional integral of
order $\alpha$ is defined by
\begin{equation}
\label{201} 
I^{\alpha,\,\eta}f(t) =
\frac{1}{\Gamma(\alpha)}\int_{0}^{t}e^{-\eta(t - s)}(t -
s)^{\alpha - 1}f(s)ds,\quad {\rm with}\  f \in L^{1}([0,\,+\infty)).
\end{equation}
The exponential fractional derivative operator of order $\alpha$ is
defined by
\begin{equation}
\label{202}
\partial_{t}^{\alpha,\,\eta}f(t) =  \frac{1}{\Gamma(1 - \alpha)}
\int_{0}^{t}e^{-\eta(t - s)}(t -
s)^{-\alpha}f'(s)ds,\quad {\rm with}\  f \in W^{1,\,1}([0,\,+\infty)).
\end{equation}
Note that $\partial_{t}^{\alpha,\,\eta}f(t) = I^{1 - \alpha,\,\eta} f'(t).$
\noindent
The following results are going to be used from now:
\begin{theorem}
\cite{15}
\label{theorem51} Let $\mu$ be the function
\begin{eqnarray}
\label{203}\mu(\xi) = |\xi|^{(2\alpha - 1)/2},\quad
\xi\in\mathbb{R},\quad 0 < \alpha < 1.
\end{eqnarray}
Then the relation between the {\it Input} $U$ and the {\it Output}
${\it O}$ is given by the following system
\begin{align}
\label{204}& \varphi_{t}({\pmb{\color{black}{x}}},\,t,\,\xi) + \xi^{2} \, \varphi({\pmb{\color{black}{x}}},\,t,\,\xi) =
\mu(\xi)U(t),\quad \xi\in\mathbb{R},\quad t > 0, \\
\label{205}& \varphi(0,\,\xi) = 0, \\
\label{206}& {\mathcal O} =
\pi^{-1}\sin(\alpha\pi)\int_{\mathbb{R}}\mu(\xi)\varphi({\pmb{\color{black}{x}}},\,t,\,\xi)
d\xi,
\end{align}
which implies that
\begin{eqnarray}
\label{207}{\mathcal O} = I^{1 - \alpha}U,\quad {\rm where}\ U\in C([0,\,+\infty)).
\end{eqnarray}
\end{theorem}
\noindent
The strategy for achieving our target is related to the elimination of fractional derivatives in time from the domain condition in the system \eqref{104}. To do this, setting $\mu(\xi) = |\xi|^{(2\alpha - 1)/2},$ $\xi\in\mathbb{R},$ $\mathfrak{C} = \pi^{-1}\sin(\alpha\pi),$ and exploiting the technique from \cite{huang}, we reduce \eqref{104} to the system \eqref{209}. In fact, we will introduce the equation
\begin{align*}
& \varphi_{t}({\pmb{\color{black}{x}}},\,t,\,\xi) + (\xi^{2} + \eta)\varphi({\pmb{\color{black}{x}}},\,t,\,\xi) - u({\pmb{\color{black}{x}}},\,t)\mu(\xi) = 0,\quad \xi\in\mathbb{R},\  \eta\geq 0,\  t > 0, \\
& \varphi({\pmb{\color{black}{x}}},\,0,\,\xi) = 0,
\end{align*}
where $\mu(\xi) = |\xi|^{(2\alpha - 1)/2}.$ Multiplying the above equation by $e^{(\xi^{2} + \eta)t}$ and integrating we get
\begin{eqnarray*}
\varphi(t,\,\xi) = \int_{0}^{t}\mu(\xi)u({\pmb{\color{black}{x}}},\,t)e^{-(\xi^{2} + \eta)(t - s)}ds
\end{eqnarray*}
and
\begin{eqnarray*}
\int_{\mathbb{R}}\mu(\xi)\varphi({\pmb{\color{black}{x}}},\,t,\,\xi)d\xi=\int_{\mathbb{R}}\int_{0}^{t}\mu^{2}(\xi)u({\pmb{\color{black}{x}}},\,t)e^{-(\xi^{2} + \eta)(t - s)}dsd\xi.
\end{eqnarray*}
On the other hand, using the Fubini Theorem and recalling the definition of the Gamma function, we get that
\begin{equation}
\label{208}
\partial_{t}^{\alpha,\,\eta}u({\pmb{\color{black}{x}}},\,t) =  \mathfrak{C}\int_{\mathbb{R}}\mu(\xi)\varphi({\pmb{\color{black}{x}}},\,t,\,\xi)d\xi.
\end{equation}
For more details of this deduction see \cite{15}. Thus, we reformulate system \eqref{103} using Theorem
\ref{theorem51}, that means, this system can be included into the augmented model

\begin{eqnarray}
\label{209}
\left\lbrace
\begin{array}{l}
\rho_{1}u_{tt} - a_{11}\Delta u - a_{12}\Delta v + \alpha(u - v) + \displaystyle\mathfrak{C}\int_{\mathbb{R}}\mu(\xi)\varphi_{1}({\pmb{\color{black}{x}}},\,t,\,\xi)d\xi = 0,  \\
\rho_{2}v_{tt} - a_{12}\Delta u - a_{22}\Delta v- \alpha(u - v) + \displaystyle\mathfrak{C}\int_{\mathbb{R}}\mu(\xi)\varphi_{2}({\pmb{\color{black}{x}}},\,t,\,\xi)d\xi  = 0 \\
\varphi_{1t}({\pmb{\color{black}{x}}},\,t,\,\xi) + (\xi^{2} + \eta)\varphi_{1}({\pmb{\color{black}{x}}},\,t,\,\xi)  = \mu(\xi)u_{t}({\pmb{\color{black}{x}}},\,t), \\ 
\varphi_{2t}({\pmb{\color{black}{x}}},\,t,\,\xi) + (\xi^{2} + \eta)\varphi_{2}({\pmb{\color{black}{x}}},\,t,\,\xi)  = \mu(\xi)v_{t}({\pmb{\color{black}{x}}},\,t), \\ 
u({\pmb{\color{black}{x}}},\,0) = u_{0}({\pmb{\color{black}{x}}}),\quad u_{t}({\pmb{\color{black}{x}}},\,0) = u_{1}({\pmb{\color{black}{x}}}),\quad  {\pmb{\color{black}{x}}}\in \Omega \\
v({\pmb{\color{black}{x}}},\,0) = v_{0}({\pmb{\color{black}{x}}}),\quad v_{t}({\pmb{\color{black}{x}}},\,0) = v_{1}({\pmb{\color{black}{x}}}),\quad  {\pmb{\color{black}{x}}}\in \Omega  \\
u({\pmb{\color{black}{x}}},\,t) = v({\pmb{\color{black}{x}}},\,t) = 0,\quad \forall\,({\pmb{\color{black}{x}}},\,t)\in\partial\Omega\times (0,\,+\infty), \\
\varphi_{1}({\pmb{\color{black}{x}}},\,0,\,\xi) = 0,\quad \varphi_{2}({\pmb{\color{black}{x}}},\,0,\,\xi) = 0,\quad \forall\,\xi \in \mathbb{R},
\end{array}
\right. 
\end{eqnarray}
where $u = u({\pmb{\color{black}{x}}},\,t),$ $v = v({\pmb{\color{black}{x}}},\,t),$ are real-valued functions and $({\pmb{\color{black}{x}}},\,t,\,\xi)\in \Omega\times (0,\,+\infty)\times \mathbb{R}.$ 
 
\medskip

We shall consider the following technical lemmas. Lemma \ref{L2.3} will be used for well-posedness and Lemma \ref{L2.4} will be used for strong stability proof.

\begin{lemma}
\label{L2.3}
If $0 < \alpha < 1$ and $\eta\ge 0,$ then 
\begin{align*}	
C(\alpha,\,\eta):=\displaystyle\int_{\mathbb{R}}\dfrac{|\xi|^{2\alpha - 1}d\xi}{(\xi^{2} + \eta + 1)} < +\infty
\quad {\rm and} \quad D(\alpha,\,\eta):=\displaystyle\int_{\mathbb{R}} \dfrac{|\xi|^{2\alpha - 1}d\xi}{(\xi^2 + \eta + 1)^{2}} < +\infty.
	\end{align*}
\end{lemma}
\begin{proof} Note that 
\begin{align*}		
C(\alpha,\,\eta):=\displaystyle\int_{\mathbb{R}}\frac{|\xi|^{2\alpha - 1}d\xi}{(\xi^{2} + \eta + 1)} = \dfrac{2}{(1 + \eta)}\displaystyle\int_{0}^{+\infty}\dfrac{|\xi|^{2\alpha - 1}d\xi}{1 + \dfrac{\xi^2}{(1 + \eta)}}.
\end{align*}	
As $0<\alpha<1,$ by making a change of variable, we have that 
$$
0<C(\alpha,\,\eta):=\dfrac{1}{(1 + \eta)^{1 - \alpha}}\displaystyle\int_{1}^{+\infty}\dfrac{d\sigma}{\sigma(\sigma - 1)^{1 - \alpha}} < + \infty.
$$
Moreover, note that 
$$
D(\alpha,\,\eta):=\displaystyle\int_{\mathbb{R}} \dfrac{|\xi|^{2\alpha - 1}\,d\xi}{(\xi^2 + \eta + 1)^{2}} \le \displaystyle\int_{\mathbb{R}}\dfrac{|\xi|^{2\alpha - 1}d\xi}{(\xi^{2} + \eta + 1)} = C(\alpha,\,\eta) < +\infty. 
$$		
\end{proof}
\begin{lemma}
\label{L2.m}
If $0<\omega<1,$ $\lambda\geq 0$ and $\delta > 0,$ then  
	\begin{align*}	
		J(\lambda,\,\omega,\,\delta):=\displaystyle\int_{\mathbb{R}} \dfrac{|y|^{2\omega - 1}dy}{y^2 + \delta + \lambda} < +\infty
		\qquad {\rm and} \qquad L(\lambda,\omega,\,\delta):=\displaystyle\int_{\mathbb{R}} \dfrac{|y|^{2\omega-1}dy}{(y^2 + \delta + \lambda)^2}<+\infty.
	\end{align*}
	\end{lemma}
\begin{proof}  Analogously to the proof of Lemma \ref{L2.3}. \end{proof}
	
\begin{lemma}
\label{L2.4}
Let $0<\alpha<1.$ If $\eta>0$ and $\lambda\in \mathbb{R},$ or if $\eta = 0$ and $\lambda>0,$ then 
\begin{align*}	
E(\lambda,\,\alpha,\,\eta):= \displaystyle\int_{\mathbb{R}}\dfrac{|\xi|^{2\alpha - 1}d\xi}{(\xi^{2} + \eta + i\lambda)} < +\infty. 
\end{align*}
Furthermore, for $h\in L^{2}(\mathbb{R}; L^{2}(\Omega)),$ we have that 
\begin{align*}
H(x,\,\lambda,\,\alpha,\,\eta):= \displaystyle\int_{\mathbb{R}}\dfrac{|\xi|^{\frac{2\alpha - 1}{2}}h(x,\,\xi)d\xi}{\xi^{2} + \eta + i\lambda}\in L^{2}(\Omega). 
\end{align*}
\end{lemma}
\begin{proof}
Note that $E(\lambda,\,\alpha,\,\eta) = F(\lambda,\,\alpha,\,\eta) + i\lambda G(\lambda,\,\alpha,\,\eta),$ where 
\begin{align*}
F(\lambda,\,\alpha,\,\eta):=\displaystyle\int_{\mathbb{R}}\dfrac{(\xi^{2} + \eta)|\xi|^{2\alpha - 1}d\xi}{\lambda^{2} + (\xi^{2} + \eta)^{2}}\quad {\rm and}\quad G(\lambda,\,\alpha,\,\eta):=\displaystyle\int_{\mathbb{R}}\dfrac{|\xi|^{2\alpha - 1}d\xi}{\lambda^{2} + (\xi^{2} + \eta)^{2}}.
\end{align*}
Using that
\begin{equation*}
G(\lambda,\,\alpha,\,\eta) = 2\int_{0}^{1}\dfrac{|\xi|^{2\alpha - 1}d\xi}{\lambda^{2} + (\xi^{2} + \delta)^{2}} + 2\int_{1}^{+\infty}\dfrac{|\xi|^{2\alpha - 1}d\xi}{\lambda^2 + (\xi^{2} + \alpha)^{2}}.
\end{equation*}
Since in both cases, ($\eta > 0$ and $\lambda\in \mathbb{R}$) or ($\eta=0$ and $\lambda>0$), we obtain
$$
\dfrac{|\xi|^{2\alpha - 1}}{\lambda^2 + (\xi^{2} + \eta)^2}\sim \dfrac{|\xi|^{2\alpha-1}}{\lambda^2 + \eta^2}\quad \mbox{as}\quad |\xi|\to 0\quad \mbox{and}\quad \dfrac{|\xi|^{2\alpha - 1}}{\lambda^{2} + (\xi^{2} + \eta)^{2}}\sim \dfrac{1}{|\xi|^{5 - 2\alpha}}\quad \mbox{as}\quad |\xi|\to +\infty,
$$
it follows that $G(\lambda,\,\eta) < + \infty.$ In a similar, 
\begin{equation*}
F(\lambda,\,\alpha,\,\eta) = 2\int_{0}^{1}\dfrac{(\xi^{2} + \eta)|\xi|^{2\alpha - 1}d\xi}{\lambda^{2} + (\xi^{2} + \eta)^{2}} + 2\int_{1}^{+\infty}\dfrac{(\xi^{2} + \eta)|y|^{2\alpha - 1}d\xi}{\lambda^2 + (\xi^{2} + \alpha)^{2}},
\end{equation*}		
and, if ($\eta>0$ and $\lambda\in \mathbb{R}$) or  ($\eta = 0$ and $\lambda>0$), we obtain 
$$
\dfrac{(\xi^{2} + \eta)|\xi|^{2\alpha - 1}}{\lambda^{2} + (\xi^{2} + \eta)^{2}}\sim \dfrac{(\xi^{2} + \eta)|\xi|^{2\alpha-1}}{\lambda^{2} + \eta^{2}}\quad \mbox{for}\quad |\xi|\to 0 
$$
and 
$$
\dfrac{(\xi^{2} + \eta)|\xi|^{2\alpha - 1}}{\lambda^{2} + (\xi^{2} + \eta)^{2}}\sim \dfrac{1}{|\xi|^{3 - 2\alpha}}\quad \mbox{for}\quad |\xi|\to +\infty.
$$	
Thus, $F(\lambda,\,\alpha,\,\eta) < +\infty,$ and consequently, it follows that $E(\lambda,\,\alpha,\,\eta) < +\infty.$ Moreover, from the Cauchy-Schwarz inequality and the fact that $h\in L^{2}(\mathbb{R}; L^{2}(\Omega)),$ it follows that	
\begin{equation*}
\int_{\Omega}\left|H(x,\,\lambda,\,\alpha,\,\eta)\right|^{2}d{\pmb{\color{black}{x}}} = \left(\int_{\mathbb{R}}\dfrac{|\xi|^{2\alpha - 1}d\xi}{\lambda^{2} + (\xi^{2} + \eta)^{2}}\right)\int_{\Omega}\int_{\mathbb{R}}|h({\pmb{\color{black}{x}}},\,\xi)|^{2}d\xi d{\pmb{\color{black}{x}}} < +\infty. 
\end{equation*}
\end{proof}
We close this section with an important functional analysis result that will be very important in this article.

\begin{theorem}(Fredholm alternative \cite{Brezis})
\label{Fredholm}
	Let $X$ be a Banach space. If $\mathcal{L}:X\to X$ is a compact linear operator on $X,$ then
	\begin{enumerate}
		\item [(i)] $\ker\left(I - \mathcal{L}\right)$ is finite dimension.
		\item [(ii)] $\left(I - \mathcal{L}\right)(X)$ is closed.
		\item [(iii)] $\ker\left(I - \mathcal{L}\right)=\{0\}\Longleftrightarrow \left(I -\mathcal{L}\right)(X) = X.$
	\end{enumerate} 
\end{theorem}

\section{Setting of the Semigroup} \label{sec3}
\label{sgsetting}
In this section we establish the well-posedness of system \eqref{209}. We denote by $L^{2}(\Omega)$ the classical space of square integrable functions over $\Omega,$ with an abuse of writing for vector valued functions. Let $\|\cdot\|$ be the standard $L^{2}$-norm over $\Omega,$ induced by the scalar product. Recall also $H_{0}^{1}(\Omega),$ the classical homogeneous Hilbert space. We thus define the phase space associated with our set of equations \eqref{209} by 
\begin{align}
\label{301}
\mathcal{H} =  \left[H_{0}^{1}(\Omega)\right]^{2}\times [L^{2}(\Omega)]^{2} \times [L^{2}(\mathbb{R};\,L^{2}(\Omega))]^{2},
\end{align}
where $\mathbb{U} = (u,\,v,\,U,\,V,\,\varphi_{1},\,\varphi_{2})^{T},$ which is a Hilbert space endowed with the following 
inner product given by 
\begin{align}
\langle \mathbb{U},\,\widetilde{\mathbb{U}}\rangle_{{\mathcal H}} = &\ \rho_{1}\langle U,\,\widetilde{U} \rangle_{L^{2}(\Omega)} + \rho_{2}\langle V,\,\widetilde{V} \rangle_{L^{2}(\Omega)} + a_{11}\langle \nabla u,\,\nabla\widetilde{u}\rangle_{L^{2}(\Omega)} + a_{22}\langle \nabla v,\,\nabla\widetilde{v}\rangle_{L^{2}(\Omega)} \nonumber \\
& + a_{12}\left[\langle \nabla u,\,\nabla\widetilde{v}\rangle_{L^{2}(\Omega)} + \langle \nabla v,\,\nabla\widetilde{u}\rangle_{L^{2}(\Omega)}  \right] + \alpha\langle u - v,\,\widetilde{u} - \widetilde{v}\rangle_{L^{2}(\Omega)}\nonumber \\
& + \mathfrak{C}
\langle\varphi_{1},\,\widetilde{\varphi}_{1}\rangle_{L^{2}\left(\mathbb{R}; L^{2}(\Omega)\right)} + \mathfrak{C}
\langle\varphi_{2},\,\widetilde{\varphi}_{2}\rangle_{L^{2}\left(\mathbb{R}; L^{2}(\Omega)\right)}, 
\end{align}
where $
\mathbb{U} = (u,\,v,\,U,\,V,\,\varphi_{1},\,\varphi_{2})^{T}$ and $\widetilde{\mathbb{U}}
= (\widetilde{u},\,\widetilde{v},\,\widetilde{U},\,\widetilde{V},\,\widetilde{\varphi}_{1},\,\widetilde{\varphi}_{2})^{T}.$ The norm is given by
\begin{align*}
\|\mathbb{U}\|_{\mathcal H}^{2} = &\left[\rho_{1}\int_{\Omega}|u_{t}|^{2}d{\pmb{\color{black}{x}}} + \rho_{2}\int_{\Omega}|v_{t}|^{2}d{\pmb{\color{black}{x}}} + \int_{\Omega}X^{T}AXd{\pmb{\color{black}{x}}} + \alpha\int_{\Omega}|u - v|^{2}d{\pmb{\color{black}{x}}} \right.        \nonumber \\
& \left.\quad +\ \mathfrak{C}\|\varphi_{1}(t)\|_{L^{2}(\mathbb{R}; L^{2}(\Omega))}^{2} + \mathfrak{C}\|\varphi_{2}(t)\|_{L^{2}(\mathbb{R}; L^{2}(\Omega))}^{2}\right],
\end{align*}
where
\begin{eqnarray*}
X^{T}{\bf A}X = \left(
\begin{array}{cc}
\nabla u & \nabla v 
\end{array}
\right)\left(
\begin{array}{cc}
a_{11} & a_{12} \\
a_{12} & a_{22} \\
\end{array}
\right)\left(
\begin{array}{c}
\nabla u \\
\\
\nabla v \\
\end{array}
\right) = a_{11}|\nabla u|^{2} + 2a_{12}\nabla u\cdot\nabla v + a_{22}|\nabla v|^{2}  \geq 0.
\end{eqnarray*}

We now wish to transform the initial boundary value problem
\eqref{209} to an abstract problem in the Hilbert space
$\mathcal{H}.$ We introduce the functions $u_{t} = U,\,$ $v_{t} = V,\,$ and rewrite the system \eqref{209} under the form of an abstract evolution problem
\begin{eqnarray}
\label{303}\frac{d}{d{t}}\mathbb{U}(t) = \mathcal{A}\mathbb{U}(t),\quad \mathbb{U}(0) = \mathbb{U}_{0},\quad
\forall\,t > 0,
\end{eqnarray}
where 
$
\mathbb{U}=(u,\,v,\,U,\,V,\,\varphi_{1},\,\varphi_{2})^{T}\ {\rm and}\
\mathbb{U}_{0} = (u_{0},\,v_{0},\,u_{1},\,v_{1},\,0,\,0)^{T},
$
and the operator $\,\mathcal{A}$ is an unbounded linear operator defined as follows $\,\mathcal{A}:\mathcal{D}(\mathcal{A})\subset
\mathcal{H}\rightarrow \mathcal{H}$ with
\begin{equation}
\label{304}\mathcal{A}\left(
\begin{array}{c}
u \\
v \\
U \\
V \\
\varphi_{1} \\
\varphi_{2} 
\end{array}
\right) =\left(
\begin{array}{c} 
U \\
V \\
\displaystyle \frac{1}{\rho_{1}}\left[a_{11}\Delta u + a_{12}\Delta v - \alpha(u - v) - \mathfrak{C}\displaystyle\int_{\mathbb{R}}\mu(\xi)\varphi_{1}(\xi)d\xi \right]\\
\displaystyle  \frac{1}{\rho_{2}}\left[a_{12}\Delta u + a_{22}\Delta v + \alpha(u - v) - \mathfrak{C}\displaystyle\int_{\mathbb{R}}\mu(\xi)\varphi_{2}(\xi)d\xi \right]  \\
-(\xi^{2} + \eta)\varphi_{1}(\xi) + \mu(\xi)U \\
-(\xi^{2} + \eta)\varphi_{2}(\xi) + \mu(\xi)V 
\end{array}
\right)
\end{equation}
with domain 
\begin{align*}
\mathcal{D}(\mathcal{A}) = &
\left\{\mathbb{U} = (u,\,v,\,U,\,V,\,\varphi_{1},\,\varphi_{2})^{T}\in {\mathcal H}: \, U,\,V\in H_{0}^{1}(\Omega),\right.  \\
& \qquad a_{11}u + a_{12}v \in H^{2}(\Omega)\ {\rm and}\  a_{12}u + a_{22}v \in H^{2}(\Omega) \\
& \qquad a_{11}u + a_{12}v - \mathfrak{C}\displaystyle\int_{\mathbb{R}}\mu(\xi)\varphi_{1}(\xi)d\xi \in L^{2}(\Omega),\  a_{12}u + a_{22}v - \mathfrak{C}\displaystyle\int_{\mathbb{R}}\mu(\xi)\varphi_{2}(\xi)d\xi \\
& \qquad \xi\varphi_{1}\in L^{2}(\mathbb{R};\,L^{2}(\Omega)),\,-(\xi^{2} + \eta)\varphi_{1} + \mu(\xi)U \in
L^{2}\left(\mathbb{R};\,L^{2}(\Omega)\right), \\
& \qquad \xi \varphi_{2}\in L^{2}(\mathbb{R};\,L^{2}(\Omega)),\,-(\xi^{2} + \eta)\varphi_{2} + \mu(\xi)V \in
L^{2}\left(\mathbb{R};\,L^{2}(\Omega)\right)\}.
\end{align*}

Let the self-adjoint and strictly positive  operator $\, A :\mathcal{D}(A)\subset
H\rightarrow H$ which is given by
\begin{equation}
\label{304b} 
A\left(
\begin{array}{c}
u \\
v 
\end{array}
\right) =\left(
\begin{array}{c} 
- \displaystyle \frac{1}{\rho_{1}}\left[a_{11}\Delta u + a_{12}\Delta v - \alpha(u - v) \right]\\
- \displaystyle  \frac{1}{\rho_{2}}\left[a_{12}\Delta u + a_{22}\Delta v + \alpha(u - v) \right]  
\end{array}
\right)
\end{equation}
with the domain
$$\mathcal{D}(A) = 
\left\{\mathbb{U} = (u,\,v)^{T}\in H: \, a_{11}u + a_{12}v \in H^{2}(\Omega) \cap H^1_0(\Omega)\ {\rm and}\  a_{12}u + a_{22}v \in H^{2}(\Omega)  \cap H^1_0(\Omega) \right\}.
$$  
and $H = L^2(\Omega) \times L^2(\Omega)$.

\medskip

So, the problem (\ref{208}) can be rewritten as in \cite{kais,kaisI}:

\begin{equation}
\label{abstract}
\left\{
\begin{array}{ll}
 Z_{\bf tt}({\bf t}) + AZ({\bf t}) + BB^* \partial_{\bf t}^{\alpha,\eta}Z ({\bf t})= 0, \ {\bf t} > 0 \\
Z(0) = Z_{0}, Z_{\bf t}(0) =  Z_{1},
\end{array}
\right.
\end{equation}

where $B = B^*=I_{H}, Z({\bf t}) = (u({\bf t}),\,v({\bf t})^{T}$ and
$Z_{0} = (u_{0},\,v_{0})^{T},\quad Z_{1} = (u_{1},\,v_{1})^{T}.$

\medskip

We define 

\begin{equation}
\label{301bis}
\mathcal{H}_{0}  = 
H_{0}^1(\Omega) \times H_{0}^{1}(\Omega) \times L^{2}(\Omega) \times L^{2}(\Omega)   
\end{equation}
equipped with the inner product given by
\begin{align}
\langle {\mathcal U},\,\tilde{{\mathcal U}}\rangle_{{\mathcal H}_{0}} = &\ \rho_{1}\langle U,\,\widetilde{U} \rangle_{L^{2}(\Omega)} + \rho_{2}\langle V,\,\widetilde{V} \rangle_{L^{2}(\Omega)} + a_{11}\langle \nabla u,\,\nabla\widetilde{u}\rangle_{L^{2}(\Omega)} + a_{22}\langle \nabla v,\,\nabla\widetilde{v}\rangle_{L^{2}(\Omega)} \nonumber \\
& + a_{12}\left[\langle \nabla u,\,\nabla\widetilde{v}\rangle_{L^{2}(\Omega)} + \langle \nabla v,\,\nabla\widetilde{u}\rangle_{L^{2}(\Omega)}  \right] + \alpha\langle u - v,\,\widetilde{u} - \widetilde{v}\rangle_{L^{2}(\Omega)}\nonumber \\
& + \mathfrak{C}
\langle\varphi_{1},\,\widetilde{\varphi}_{1}\rangle_{L^{2}\left(\mathbb{R}; L^{2}(\Omega)\right)}, 
\label{302bis}
\end{align}
where ${\mathcal U}=(u,\,v,\,U,\,V)^{T}$ and $\ \tilde{{\mathcal U}}=(\tilde{u},\,\tilde{v},\,\tilde{U},\,\tilde{V})^{T}.$

\medskip

Then, the operator 
$\,\mathcal{A}_{0}:\mathcal{D}(\mathcal{A}_{0})\subset
\mathcal{H}_{0}\rightarrow \mathcal{H}_{0}$ given by
\begin{equation}
\label{304bis}\mathcal{A}_{0} \left(
\begin{array}{c}
u \\
\\
v\\
\\
U \\
\\
V
\end{array}
\right) =\left(
\begin{array}{c}
U \\
\\
V \\
\\
\displaystyle \frac{1}{\rho_{1}}\left[a_{11}\Delta u + a_{12}\Delta v - \alpha(u - v) \right] - U\\
\displaystyle  \frac{1}{\rho_{2}}\left[a_{12}\Delta u + a_{22}\Delta v + \alpha(u - v) \right] - V
\end{array}
\right).
\end{equation}
with the domain
\begin{align*}
\mathcal{D}({\mathcal A}_{0}) =  \left\{\mathbb{U} = (u,\,v,\,U,\,V)^{T}\in {\mathcal H}_0:\ U,\,V\in H_{0}^1(\Omega),\ (u,v) \in \mathcal{D}(A)\right\},
\end{align*}
 generates a C$_{0}$-semigroup of contractions in ${\mathcal H}_{0}, (e^{t\mathcal{A}_0})_{t \geq 0}.$ Moreover, the following auxiliary problem:

\begin{align}
\label{104bis}(u,\,v)({\pmb{\color{black}{x}}},\,0) = (u_{0}({\pmb{\color{black}{x}}}),\,v_{0}({\pmb{\color{black}{x}}})),\quad (u_{t},\,v_{t})({\pmb{\color{black}{x}}},\,0) = (u_{1}({\pmb{\color{black}{x}}}),\,v_{1}({\pmb{\color{black}{x}}})).
\end{align}
For the sake of simplicity, we consider a homogeneous material, so that all parameters $\rho,\,j$ are constants equal to $1.$ Following the idea given by \cite{kais} we study the asymptotic behavior for the system
\begin{eqnarray}
\left\lbrace
\label{105bis}
\begin{array}{l}
\rho_{1}u_{tt} - a_{11}\Delta u - a_{12}\Delta v + \alpha(u - v) + u_t = 0, \quad ({\pmb{\color{black}{x}}},\,t)\in \Omega\times (0,\,\infty),  \\
\rho_{2}v_{tt} - a_{12}\Delta u - a_{22}\Delta v - \alpha(u - v)  + v_t  = 0, \quad ({\pmb{\color{black}{x}}},\,t)\in \Omega\times (0,\,\infty),\\
u({\pmb{\color{black}{x}}},\,0) = u_{0}({\pmb{\color{black}{x}}}),\quad u_{t}({\pmb{\color{black}{x}}},\,0) = u_{1}({\pmb{\color{black}{x}}}),\quad  {\pmb{\color{black}{x}}}\in \Omega, \\
v({\pmb{\color{black}{x}}},\,0) = v_{0}({\pmb{\color{black}{x}}}),\quad v_{t}({\pmb{\color{black}{x}}},\,0) = v_{1}({\pmb{\color{black}{x}}}),\quad  {\pmb{\color{black}{x}}}\in \Omega,  \\
u({\pmb{\color{black}{x}}},\,t) = v({\pmb{\color{black}{x}}},\,t) = 0,\quad \forall\,({\pmb{\color{black}{x}}},\,t)\in\partial\Omega\times (0,\,+\infty),
\end{array}
\right. 
\end{eqnarray}
 admits a unique solution $(u(x,t),\,v(x,t))$ such that if $(u_{0},\,v_{0},\,u_{1},\,v_{1})\in {\mathcal D}({\mathcal A}_{0})$, then the solution $(u(x,\,{\bf t}),\,v(x,\,{\bf t})$ of \eqref{105bis} verifies the following regularity property:
\begin{align*}
{\mathcal U}=(u,\,v,\,u_{t},\,v_{t})\in C([0,\,+\infty),\,{\mathcal D}({\mathcal A}_{0}))\cap C^{1}([0,\,+\infty),\,{\mathcal H}_{0}),
\end{align*}
and when $(u_{0},\,v_{0},\,u_{1},\,v_{1})\in {\mathcal H}_{0}$, then 
\begin{align*}
{\mathcal U}=(u,v,u_{t},v_{t})\in C([0,\,+\infty),\,{\mathcal H}_{0}).
\end{align*}

So, according to \cite{kais,kaisI}, we have the following proposition: 
\begin{proposition}
\label{lu}
The operator $\mathcal{A}$ is the infinitesimal generator of a contraction semigroup $\{{\mathcal S}(t)\}_{t\geq 0}.$
\end{proposition}
 
From the Proposition \ref{lu}, the system \eqref{209} is well-posed in the energy space ${\mathcal H}$ and we have the following theorem:
\begin{align}
\label{312} \mathbb{U}=(u,\,v,\,U,\,V,\,\varphi_{1},\,\varphi_{2})^{T}\quad
{\rm and} \quad
\mathbb{U}_{0}=(u_{0},\,v_{0},\,u_{1},\,v_{1},\,0,\,0)^{T},
\end{align}
\begin{theorem}(Existence and uniqueness of solutions)
\label{exis}
If $(u_{0},\,v_{0},\,u_{1},\,v_{1},\,0,\,0)\in {\mathcal H},$ 
the problem \eqref{209} admits a unique solution
\begin{equation*}
(u,\,v,\,u_{t},\,v_{t},\,\varphi_{1},\,\varphi_{2}) \in C\left([0,\,+\infty); {\mathcal H}\right),
\end{equation*}
and for $(u_{0},\,v_{0},\,u_{1},\,v_{1},\,0,\,0)\in {\mathcal D}({\mathcal A})$, the problem \eqref{209} admits a unique solution
\begin{equation*}
(u,\,v,\,u_{t},\,v_{t},\,\varphi_{1},\,\varphi_{2})\in C\left([0,\,+\infty); {\mathcal D}({\mathcal A})\right)\cap C^{1}\left([0,\,+\infty); {\mathcal H}\right).
\end{equation*}
Moreover, the energy in time $t\geq 0$ is given by
\begin{align}
E(t) = &\ \frac{1}{2}\left[\rho_{1}\int_{\Omega}|u_{t}|^{2}d{\pmb{\color{black}{x}}} + \rho_{2}\int_{\Omega}|v_{t}|^{2}d{\pmb{\color{black}{x}}} + \int_{\Omega}X^{T}AXd{\pmb{\color{black}{x}}} + \alpha\int_{\Omega}|u - v|^{2}d{\pmb{\color{black}{x}}} \right. \nonumber \\
\label{313}& \left.\qquad\quad +\ \mathfrak{C}\|\varphi_{1}(t)\|_{L^{2}(\mathbb{R}; L^{2}(\Omega))}^{2} + \mathfrak{C}\|\varphi_{2}(t)\|_{L^{2}(\mathbb{R}; L^{2}(\Omega))}^{2}\right].
\end{align}
Furthermore, satisfies
\begin{align}
\label{314}\frac{d}{dt}E(t) = &- \mathfrak{C}\int_{\mathbb{R}}(\xi^{2} + \eta)\|\varphi_{1}(t,\,\xi)\|_{L^{2}(\Omega)}^{2}d\xi - \mathfrak{C}\int_{\mathbb{R}}(\xi^{2} + \eta)\|\varphi_{2}(t,\,\xi)\|_{L^{2}(\Omega)}^{2}d\xi. 
\end{align}
\end{theorem}
 
\section{Strong stability} \label{sec4}
\label{ststab}
In this section, we prove that the solutions of system \eqref{303} converge asymptotically to zero. The following theorem plays an important role.
\begin{theorem}(Arendt-Batty \cite{Arendt})
\label{Arendt}
	Let ${\mathcal A}$ be the generator of a C$_{0}$-semigroup $({\mathcal S}(t))_{t\ge 0}$ in a reflexive Banach space $X.$ If the following conditions are satisfied:
	\begin{enumerate}
		\item [(i)] ${\mathcal A}$ has no purely imaginary eigenvalues,
		\item [(ii)] $\sigma({\mathcal A})\cap i\mathbb{R}$ is countable,
	\end{enumerate}
	then, $\{{\mathcal S}(t)\}_{t\ge 0}$ is strongly stable.
\end{theorem}
\begin{proposition}
\label{P4.1}
If $\lambda \in \mathbb{R},$ then $i\lambda I - {\mathcal A}$ is injective.
\end{proposition}
\begin{proof} 
Let $\lambda\in \mathbb{R}$ Such that $i\lambda$ is an eigenvalue of the operator ${\mathcal A},$ and let $\mathbb{U} = (u,\,v,\,U,\,V,\,\varphi_{1},\,\varphi_{2})^{T}\in \mathcal{D}({\mathcal A})$ be the associated eigenvector. Then ${\mathcal A}\mathbb{U} = i\lambda \mathbb{U}.$ Equivalently 	
\begin{align}
\label{401}
\begin{cases}
i\lambda u - U = 0  \iff U = i\lambda u  \\
i\lambda v - V = 0 \iff V = i\lambda v \\
\displaystyle  a_{11}\Delta u + a_{12}\Delta v - \alpha(u - v) + \mathfrak{C}\displaystyle\int_{\mathbb{R}}\mu(\xi)\varphi_{1}(\xi)d\xi =  i\lambda\rho_{1}U,  \\
\displaystyle a_{12}\Delta u + a_{22}\Delta v + \alpha(u - v) + \mathfrak{C}\displaystyle\int_{\mathbb{R}}\mu(\xi)\varphi_{2}(\xi)d\xi =  i\lambda\rho_{2}V,   \\
(\xi^{2} + \eta + i\lambda)\varphi_{1}(\xi) = \mu(\xi)U,  \quad \forall \,\xi\in \mathbb{R}, \\
(\xi^{2} + \eta + i\lambda)\varphi_{2}(\xi) = \mu(\xi)V,  \quad \forall \,\xi\in \mathbb{R}.
\end{cases}
\end{align}
Note that
\begin{align*}
0 = &\ Re\left\langle\mathcal{A}\mathbb{U},\,\mathbb{U}\right\rangle_{\mathcal H} = 
- \mathfrak{C}\int_{\mathbb{R}}(\xi^{2} +
\eta)\|\varphi_{1}(\xi)\|_{L^{2}(\Omega)}^{2}d\xi - \mathfrak{C}\int_{\mathbb{R}}(\xi^{2} +
\eta)\|\varphi_{2}(\xi)\|_{L^{2}(\Omega)}^{2}d\xi. 
\end{align*}
Therefore
\begin{eqnarray}
\label{402}
\begin{cases}
\varphi_{1}({\pmb{\color{black}{x}}},\,\xi) = 0\quad \mbox{a.\ e.}\quad {\rm in}\quad ({\pmb{\color{black}{x}}},\,\xi) \in \Omega\times\mathbb{R}, \\
\varphi_{2}({\pmb{\color{black}{x}}},\,\xi) = 0\quad \mbox{a.\ e.}\quad {\rm in}\quad ({\pmb{\color{black}{x}}},\,\xi) \in \Omega\times\mathbb{R}.
\end{cases}
\end{eqnarray}  
Applying \eqref{402}$_{1,\,2}$ to \eqref{401}$_{5,\,6}$ respectively, we obtain:
\begin{eqnarray}
\label{403}
\begin{cases}
U({\pmb{\color{black}{x}}}) = 0 \quad  \mbox{a.\ e.}\quad {\rm in}\quad {\pmb{\color{black}{x}}}\in \Omega, \\
V({\pmb{\color{black}{x}}}) = 0 \quad  \mbox{a.\ e.}\quad {\rm in}\quad {\pmb{\color{black}{x}}}\in\Omega. \\
\end{cases}
\end{eqnarray}
Now, applying \eqref{403}$_{1,\,2}$ to \eqref{401}$_{1,\,2}$ respectively, we have
\begin{eqnarray}
\label{404}
\begin{cases}
i\lambda u({\pmb{\color{black}{x}}}) = 0 \quad  \mbox{a.\ e.}\quad {\rm in}\quad {\pmb{\color{black}{x}}}\in \Omega, \\
i\lambda v({\pmb{\color{black}{x}}}) = 0 \quad  \mbox{a.\ e.}\quad {\rm in}\quad {\pmb{\color{black}{x}}}\in \Omega. 
\end{cases}
\end{eqnarray}
If $\lambda\neq 0$, then  
\begin{eqnarray*}
\begin{cases}
u = 0  \quad  \mbox{a.\ e.\ \ on}\quad  \Omega, \\
v = 0 \quad  \mbox{a.\ e.\ \ on}   \quad \Omega. 
\end{cases}
\end{eqnarray*}
Thus, from \eqref{401}$_{3,\,4}$ and \eqref{402}$_{1,\,2}$ it follows that 
\begin{align*}
\begin{cases}
\displaystyle a_{11}\Delta u + a_{12}\Delta v - \alpha(u - v) =  0,  \\
\displaystyle a_{12}\Delta u + a_{22}\Delta v + \alpha(u - v) =  0,  
\end{cases}
\end{align*}
almost everywhere on $\Omega.$
From the boundary conditions, it follows that $u\equiv 0,$ $v\equiv 0,$  and therefore, the fourth equation implies that $u({\pmb{\color{black}{x}}}) = 0$ almost everywhere in ${\pmb{\color{black}{x}}}\in\Omega.$ Assuming that $\lambda = 0,$ from the third and fourth equations of the system \eqref{401}, along with the boundary conditions of the problem, we obtain the following system
\begin{eqnarray}
\left\lbrace
\label{405}
\begin{array}{l}
\displaystyle a_{11}\Delta u + a_{12}\Delta v - \alpha(u - v) =  0,\quad {\pmb{\color{black}{x}}}\in\Omega,\quad t > 0, \\
\displaystyle a_{12}\Delta u + a_{22}\Delta v + \alpha(u - v) =  0, \quad {\pmb{\color{black}{x}}}\in\Omega,\quad t > 0, \\
u_{t}({\pmb{\color{black}{x}}},\,t) = v_{t}({\pmb{\color{black}{x}}},\,t) = 0,\quad {\pmb{\color{black}{x}}}\in\Omega,\quad t > 0, \\
u({\pmb{\color{black}{x}}},\,t) = 0,\quad  v({\pmb{\color{black}{x}}},\,t) = 0,\quad {\pmb{\color{black}{x}}}\in \partial\Omega\quad t > 0, \\
U({\pmb{\color{black}{x}}},\,t)  = 0,\quad  V({\pmb{\color{black}{x}}},\,t) = 0,\quad {\pmb{\color{black}{x}}}\in \partial\Omega\quad t > 0. 
\end{array}
\right. 
\end{eqnarray}
Applying the operator method to the  system \eqref{405}, one obtain that
$u \equiv 0$ and $v \equiv 0.$ Therefore, in any case, $\ker(i\lambda I - {\mathcal A}) = \{0\}.$ 
\end{proof}

\begin{corollary}
\label{C4.2}
If $\lambda \in \mathbb{R},$ then $i\lambda$ is not an eigenvalue of ${\mathcal A}$.
\end{corollary}
\begin{proposition}
\label{p4.2}
If $\eta = 0,$ then the operator ${\mathcal A}$ is not invertible, and consequently $0\in \sigma({\mathcal A}).$
\end{proposition}
\begin{proof} For $\Omega = (0,L)$, let $Z_{0}=\left(\sin(\pi x/L),\,0,\,0,\,0,\,0,\,0\right)\in \mathcal{H}$ and we assume that there exists $\mathbb{U}_{0} = (u_{0},\,v_{0},\,U_{0},\,V_{0},\,\varphi_{1,\,0},\,\varphi_{2,\,0})\in \mathcal{D}({\mathcal A})$ such that ${\mathcal A}\mathbb{U}_{0} = Z_{0}.$ In this case, $\varphi_{1,0}(\xi) = |\xi|^{\frac{2\alpha - 5}{2}}\sin(\pi x/L).$ However, $\varphi_{1,0}\notin L^{2}(\mathbb{R}; L^{2}(\Omega))$ for $0 < \alpha < 1.$ 

\end{proof}

\begin{proposition}
	\label{P4.3}
	\begin{enumerate}
		\item [(a)] If $\eta = 0,$ then $i\lambda I - {\mathcal A}$ is surjective, for any $\lambda\neq 0.$
		\item [(b)] If $\eta > 0$ and $\lambda\in \mathbb{R},$ then $i\lambda I - {\mathcal A}$ is surjective. 
	\end{enumerate}
\end{proposition}
\begin{proof} 
Given $\mathbb{F} = (f_{1},\,f_{2},\,f_{3},\,f_{4},\,f_{5},\,f_{6})^{T}\in \mathcal{H},$ we aim to show that there exists a vector $\mathbb{U} = (u,\,v,\,U,\,V,\,\varphi_{1},\,\varphi_{2})^{T}\in {\mathcal D}({\mathcal A})$ such that $(i\lambda I - {\mathcal A})\mathbb{U} = \mathbb{F}.$ That is,
\begin{align}
\label{406}
\begin{cases}
i\lambda u - U = f_{1} \iff U =  i\lambda u - f_{1} \\
i\lambda v - V = f_{2}  \iff V =  i\lambda v - f_{2} \\
i\lambda\rho_{1}U - \displaystyle a_{11}\Delta u - a_{12}\Delta v + \alpha (u - v) + \mathfrak{C}\displaystyle\int_{\mathbb{R}}\mu(\xi)\varphi_{1}(\xi)d\xi  = \rho_{1}f_{3},  \\
i\lambda \rho_{2}V - \displaystyle a_{12}\Delta u - a_{22}\Delta v - \alpha (u - v) + \mathfrak{C}\displaystyle\int_{\mathbb{R}}\mu(\xi)\varphi_{2}(\xi)d\xi  = \rho_{2}f_{4}, \\
(|\xi|^{2} + \eta + i\lambda)\varphi_{1}(\xi) - \mu(\xi)U = f_{5},  \quad \forall \,\xi\in \mathbb{R}, \\
(|\xi|^{2} + \eta + i\lambda)\varphi_{2}(\xi) - \mu(\xi)V = f_{6},  \quad \forall \,\xi\in \mathbb{R}. 
\end{cases}
\end{align}
Replacing \eqref{406}$_{1,\,2}$ into \eqref{406}$_{5,\,6}$ we obtain 
\begin{align}
\label{407}
\begin{cases}
\displaystyle\varphi_{1}(\xi) = \frac{f_{5}(\xi)}{(|\xi|^{2} + \eta + i\lambda)} - \frac{\mu(\xi)f_{1}}{(|\xi|^{2} + \eta + i\lambda)} + \frac{i\lambda\mu(\xi)u}{(|\xi|^{2} + \eta + i\lambda)}, \\
\displaystyle\varphi_{2}(\xi) = \frac{f_{6}(\xi)}{(|\xi|^{2} + \eta + i\lambda)} - \frac{\mu(\xi)f_{2}}{(|\xi|^{2} + \eta + i\lambda)} + \frac{i\lambda\mu(\xi)v}{(|\xi|^{2} + \eta + i\lambda)}.
\end{cases} 
\end{align}
From Lemma $\ref{L2.4}$, it follows that
\begin{align}
\label{408}
\begin{cases}
\displaystyle\mathfrak{C}\int_{\mathbb{R}}\mu(\xi)\varphi_{1}(\xi)d\xi
= \mathfrak{C}\left[\int_{\mathbb{R}}\dfrac{\mu(\xi)f_{5}(\xi)d\xi}{(|\xi|^{2} + \eta + i\lambda)} + E_{1}(\lambda,\,\alpha,\,\eta)(i\lambda u - f_{1})\right], \\
\\
\displaystyle\mathfrak{C}\int_{\mathbb{R}}\mu(\xi)\varphi_{2}(\xi)d\xi
= \mathfrak{C}\left[\int_{\mathbb{R}}\frac{\mu(\xi)f_{6}(\xi)d\xi}{(|\xi|^{2} + \eta + i\lambda)} + E_{2}(\lambda,\,\alpha,\,\eta)(i\lambda v - f_{2})\right].
\end{cases}
\end{align}
Thus, from the remaining equations in the system \eqref{406}$_{3,\,4},$ it follows that
\begin{align*}
\begin{cases}
- \lambda^{2}\rho_{1}u - i\lambda f_{1} - \displaystyle a_{11}\Delta u - a_{12}\Delta v + \alpha (u - v) \\
\displaystyle +\ \mathfrak{C}\left[\int_{\mathbb{R}}\dfrac{\mu(\xi)f_{5}(\xi)d\xi}{(|\xi|^{2} + \eta + i\lambda)} + E_{1}(\lambda,\,\alpha,\,\eta)(i\lambda u - f_{1})\right] = \rho_{1}f_{3}, \\
\\
- \lambda^{2}\rho_{2}v - i\lambda f_{2} - \displaystyle a_{12}\Delta u - a_{22}\Delta v - \alpha (u - v) \\
\displaystyle +\ \mathfrak{C}\left[\int_{\mathbb{R}}\frac{\mu(\xi)f_{6}(\xi)d\xi}{(|\xi|^{2} + \eta + i\lambda)} + E_{2}(\lambda,\,\alpha,\,\eta)(i\lambda v - f_{2})\right]
= \rho_{2}f_{4}.
\end{cases}
\end{align*}
or
\begin{align}
\label{409}
\begin{cases}
- \lambda^{2}\rho_{1}u - i\lambda f_{1} - \displaystyle a_{11}\Delta u - a_{12}\Delta v + \alpha (u - v) \\
\displaystyle +\ \mathfrak{C}\int_{\mathbb{R}}\dfrac{\mu(\xi)f_{5}(\xi)d\xi}{(|\xi|^{2} + \eta + i\lambda)} + \mathfrak{C}E_{1}(\lambda,\,\alpha,\,\eta)(i\lambda u - f_{1}) = \rho_{1}f_{3}, \\
\\
- \lambda^{2}\rho_{2}v - i\lambda f_{2} - \displaystyle a_{12}\Delta u - a_{22}\Delta v - \alpha (u - v) \\
\displaystyle +\ \mathfrak{C}\int_{\mathbb{R}}\frac{\mu(\xi)f_{6}(\xi)d\xi}{(|\xi|^{2} + \eta + i\lambda)} + \mathfrak{C}E_{2}(\lambda,\,\alpha,\,\eta)(i\lambda v - f_{2})
= \rho_{2}f_{4}.
\end{cases}
\end{align}
Then, if $\lambda = 0,$ by hypothesis, we have $\eta >0.$ In that case, from \eqref{409} we have
\begin{align}
\label{410}
\begin{cases}
- \displaystyle a_{11}\Delta u - a_{12}\Delta v + \alpha (u - v) \displaystyle  = \rho_{1}f_{3} + \mathfrak{C}E_{1}(0,\,\alpha,\,\eta)f_{1} - \mathfrak{C}\int_{\mathbb{R}}\dfrac{\mu(\xi)f_{5}(\xi)d\xi}{(|\xi|^{2} + \eta)}, \\
\\
- \displaystyle a_{12}\Delta u - a_{22}\Delta v - \alpha (u - v)= \rho_{2}f_{4} + \mathfrak{C}E_{2}(0,\,\alpha,\,\eta)f_{2} - \mathfrak{C}\int_{\mathbb{R}}\frac{\mu(\xi)f_{6}(\xi)d\xi}{(|\xi|^{2} + \eta)} .
\end{cases}
\end{align}
Multiplying equations in \eqref{410} by $\widetilde{u}\in H_{0}^{1}(\Omega)$ and $\widetilde{v}\in H_{0}^{1}(\Omega)$ respectively, and  proceeding in a similar manner to the approach used in the proof of Proposition $\ref{lu}$, we get the problem of finding a vector $\left(u,\,v\right)\in H_{0}^{1}(\Omega)\times H_{0}^{1}(\Omega)$ such that
\begin{equation}
\label{411}
\mathcal{B}((u,\,v),\,(\widetilde{u},\,\widetilde{v})) = \mathcal{L}(\widetilde{u},\,\widetilde{v}),
\end{equation}
where $\mathcal{B}:[H_{0}^{1}(\Omega)]^{2}\times [H_{0}^{1}(\Omega)]^{2}\longrightarrow\mathbb{R}$ is the bilinear form defined by
\begin{eqnarray}
\mathcal{B}((u,\,v),\,(\widetilde{u},\,\widetilde{v})) & = & a_{11}\displaystyle\int_{\Omega}\nabla u\cdot\nabla\widetilde{u}\,d{\pmb{\color{black}{x}}} + a_{12}\displaystyle\int_{\Omega}\nabla v\cdot\nabla\widetilde{u}\,d{\pmb{\color{black}{x}}} + \alpha\int_{\Omega}u\widetilde{u}\,d{\pmb{\color{black}{x}}} - \alpha\int_{\Omega}v\widetilde{u}\,d{\pmb{\color{black}{x}}}\nonumber  \\
& & -\ a_{12}\displaystyle\int_{\Omega}\nabla u\cdot\nabla\widetilde{v}\,d{\pmb{\color{black}{x}}} + a_{22}\displaystyle\int_{\Omega}\nabla v\cdot\nabla\widetilde{v}\,d{\pmb{\color{black}{x}}} - \alpha\int_{\Omega}u\widetilde{v}\,d{\pmb{\color{black}{x}}} + \alpha\int_{\Omega}v\widetilde{v}\,d{\pmb{\color{black}{x}}}
\end{eqnarray}
and $\mathcal{L}:H_{0}^{1}(\Omega)\times H_{0}^{1}(\Omega)\longrightarrow\mathbb{R}$ is the linear form defined by
\begin{eqnarray}
\mathcal{L}(\widetilde{u},\,\widetilde{v}) = \int_{\Omega}F\,\widetilde{u}\,d{\pmb{\color{black}{x}}} + \int_{\Omega}G\,\widetilde{v}\,d{\pmb{\color{black}{x}}},
\end{eqnarray}
where 
$$
F = \rho_{1}f_{3} + \mathfrak{C}E_{1}(0,\,\alpha,\,\eta)f_{1} - \mathfrak{C}\displaystyle  \int_{\Omega}\widetilde{u}\displaystyle\int_{\mathbb{R}}\dfrac{\mu(\xi)f_{5}(\xi)d\xi}{(|\xi|^{2} + \eta)}\,d{\pmb{\color{black}{x}}}
$$
and
$$
G = \rho_{2}f_{4} + \mathfrak{C}E_{2}(0,\,\alpha,\,\eta)f_{2} - \mathfrak{C}\displaystyle\int_{\Omega}\widetilde{v}\displaystyle\int_{\mathbb{R}}\dfrac{\mu(\xi)f_{6}(\xi)d\xi}{(|\xi|^{2} + \eta)}\,d{\pmb{\color{black}{x}}}.
$$ 
However,  
\begin{align*}
\begin{cases}
\ \left|\mathfrak{C}\displaystyle\int_{\Omega}\widetilde{u}\displaystyle\int_{\mathbb{R}}\frac{\mu(\xi)f_{5}(\xi)d\xi}{|\xi|^{2} + \eta}\,d{\pmb{\color{black}{x}}}\right|\leq |\Omega|\mathfrak{C}H_{1}({\pmb{\color{black}{x}}},\,0,\,\alpha,\,\eta)\|\widetilde{u}\|_{H_{0}^{1}(\Omega)}, \\
\\
\ \left|\mathfrak{C}\displaystyle\int_{\Omega}\widetilde{v}\displaystyle\int_{\mathbb{R}}\frac{\mu(\xi)f_{6}(\xi)d\xi}{|\xi|^{2} + \eta}\,d{\pmb{\color{black}{x}}}\right|\leq |\Omega|\mathfrak{C}H_{2}({\pmb{\color{black}{x}}},\,0,\,\alpha,\,\eta)\|\widetilde{v}\|_{H_{0}^{1}(\Omega)}.
	\end{cases}
\end{align*}
$\mathcal{L}$ is continuous, and therefore it suffices utilize the Lax-Milgram Theorem.\\
Finally, suppose that $\lambda\neq 0$. Define the linear unbounded operator $\mathcal{M}:[H_{0}^{1}(\Omega)]^2\to \left([H_{0}^{1}(\Omega)]^{2}\right)^{*}$ given by
\begin{center}
$\displaystyle\mathcal{M}\left( \begin{array}{c}
u    \\
v
\end{array}  \right)
=
\left(\begin{array}{c}
- \displaystyle a_{11}\Delta u - a_{12}\Delta v + \alpha (u - v) \displaystyle + I_{1}(\lambda,\,\beta,\,\eta)u \\
\\
 - \displaystyle a_{12}\Delta u - a_{22}\Delta v - \alpha (u - v) + I_{2}(\lambda,\,\beta,\,\eta)v
\end{array}\right),$
\end{center}
where $I_{1}(\lambda,\,\alpha,\,\eta) = i\,\lambda\,u\,\mathfrak{C}\,E_{1}(\lambda,\,\alpha,\,\eta)\ $ and $\ I_{2}(\lambda,\,\alpha,\,\eta) = i\,\lambda\,v\,\mathfrak{C}\,E_{2}(\lambda,\,\alpha,\,\eta).$ 
From Lax-Milgram Theorem, it follows that it is an isomorphism. Thus, we have that the system \eqref{410} is equivalent to

\begin{equation}\label{43.11}
	\left(-\lambda^{2}\mathcal{M}^{-1} - I\right)
	\left( \begin{array}{c}
		u    \\
		v
	\end{array}  \right) = \mathcal{M}^{-1}\left( \begin{array}{c}
		\widetilde{F}    \\
		\widetilde{G}
	\end{array} \right),
\end{equation}
$$
\widetilde{F} = [i\lambda + \mathfrak{C}E_{1}(\lambda,\,\alpha,\,\eta)]f_{1} + \rho_{1}f_{3} + |\Omega|\mathfrak{C}H_{1}({\pmb{\color{black}{x}}},\,\lambda,\,\alpha,\,\eta)\|\widetilde{u}\|_{H_{0}^{1}(\Omega)}
$$ 
and 
$$\widetilde{G}=[i\lambda + \mathfrak{C}E_{2}(\lambda,\,\alpha,\,\eta)]f_{2} + \rho_{2}f_{4} + |\Omega|\mathfrak{C}H_{2}({\pmb{\color{black}{x}}},\,\lambda,\,\alpha,\,\eta)\|\widetilde{v}\|_{H_{0}^{1}(\Omega)}.
$$
Since the operator $\mathcal{M}^{-1}$ is an isomorphism and $I$ is a compact operator from $[H_{0}^{1}(\Omega)]^{2}$ to $\left([H_{0}^{1}(\Omega)]^2\right)^{*}.$ Then $\mathcal{M}^{-1}$ is compact operator from $[H_{0}^{1}(\Omega)]^{2}$ to $[H_{0}^{1}(\Omega)]^{2}.$ Consequently, by Fredholm alternative (Theorem \ref{Fredholm}), proving the existence of $(u,\,v)\in [H_{0}^{1}(\Omega)]^{2}$ solution of \eqref{43.11} reduces to proving 
$$
\ker\left(-\lambda^{2}\mathcal{M}^{-1} - I\right)
= \left\{\left( \begin{array}{c}
 	0    \\
 	0
 \end{array}  \right)\right\}.
 $$
Indeed, if $(\widetilde{u},\,\widetilde{v})^{T}\in \ker\left(-\lambda^{2}\mathcal{M}^{-1} - I\right),$ then 
$$	
\left(-\lambda^{2}I - \mathcal{M}\right)
\left( \begin{array}{c}
\widetilde{u}    \\
\widetilde{v}
\end{array} \right) = 
\left(\begin{array}{c}
	0    \\
	0
\end{array}  
\right).
$$
That is, 
\begin{align}
\label{43.12}
\begin{cases}
(I_{1}(\lambda,\,\alpha,\,\eta) - \lambda^{2})\widetilde{u} + \displaystyle a_{11}\Delta \widetilde{u} + a_{12}\Delta \widetilde{v} - \alpha (\widetilde{u} + \widetilde{v})  = 0,  \\
\\
(I_{2}(\lambda,\,\alpha,\,\eta) + \lambda^{2})\widetilde{v} + \displaystyle a_{12}\Delta \widetilde{u} - a_{22}\Delta \widetilde{v} + \alpha (\widetilde{u} - \widetilde{v}) = 0.
\end{cases}
\end{align}
Multiplying \eqref{43.12}, by $\overline{\widetilde{u}}$ and $\overline{\widetilde{v}}$ respectively, integrating by parts and using the boundary conditions, it follows that $(\widetilde{u},\,\widetilde{v})^{T} = (0,\,0)$. So, it follows that from Fredholm alternative  (Theorem \ref{Fredholm}) there is a unique solution $(u,\,v)\in [H_{0}^{1}(\Omega)]^{2}$ for \eqref{43.11}. By elliptic regularity, it follows that $u,\,v\in H_{0}^{1}(\Omega)\cap H^{2}(\Omega).$ Now just take $U = i\lambda u - f_{1},$ $\ V = i\lambda v - f_{2}$ and $\varphi_{1}(\xi),$ $\varphi_{2}(\xi)$ as given in \eqref{407}. Evidently,  $\mathbb{U} = (u,\,v,\,U,\,V,\,\varphi_{1},\,\varphi_{2})^{T}\in \mathcal{D}(\mathcal{A})$ and  $(i\lambda\mathcal{I} - \mathcal{A})\mathbb{U} = \mathbb{F}.$ 

\end{proof}

\begin{corollary}
	\label{C4.4} 	
	\begin{enumerate}
		\item [(a)] If $\eta=0,$ then $i\lambda \notin \sigma({\mathcal A}),$ for any $\lambda \in \mathbb{R}^*,$
		\item [(b)] If $\eta> 0,$ then $i\lambda \notin \sigma({\mathcal A}),$ for any $\lambda\in \mathbb{R}.$
	\end{enumerate}
\end{corollary}

\begin{theorem}
\label{T4.5}
The C$_{0}$-semigroup $\{{\mathcal S}(t)\}_{t\ge 0}$ generated by operator ${\mathcal A}$ is strongly stable in $\mathcal{H},$ i. e.,
$$
\lim_{t\to +\infty}\left\|e^{t{\mathcal A}}\mathbb{U}_{0}\right\|_{\mathcal{H}} = 0, \quad\forall \; \mathbb{U}_{0}\in \mathcal{H}.
$$
\end{theorem}
\begin{proof} From Corollary $\ref{C4.2},$ it follows that operator ${\mathcal A}$ does not have purely imaginary eigenvalues. However, if $\eta = 0,$ Proposition $\ref{p4.2}$ and item $(a)$ of the Corollary $\ref{C4.4},$ imply that $\sigma({\mathcal A}) \cap i\mathbb{R} = \{0\}.$ In the case of $\eta > 0,$ using item (b) of the Corollary $\ref{C4.4},$, we conclude that $\sigma({\mathcal A})\cap i\mathbb{R} = \emptyset.$ Therefore, in both cases, we can apply the Arendt and Batty Theorem, leading to the desired result. \end{proof}

\section{Polynomial Stability for ${\pmb{\color{black}{\eta}}} \geq 0$}  \label{sec5}

According to \cite{kais,kaisI}, we have the following:

\begin{theorem}
\label{theorem501bis}
For $\eta > 0$, we have for every $(u_0,v_0,u_1,v_1,0,0) \in{\mathcal D}({\mathcal A}),$\begin{align}
\label{501bis}
E(t) \leq \frac{C}{t^{\frac{1}{1 - \alpha}}} \, \left\|(u_0,v_0,u_1,v_1,0,0)\right\|^2_{\mathcal{D}(\mathcal{A})},\quad \forall\,t\geq 0.
\end{align}
\end{theorem}
 To determine the stability rate in the case where $\eta=0$, we rely on the following result from \cite[Theorem 8.4]{BCT}:

\begin{theorem}[\cite{BCT}] \label{TBC}
Let $\mathcal{T}(t)$ be a bounded C$_0$-semigroup on a Hilbert space $X$ with generator $H.$ Assume that $\sigma (H) \cap i \mathbb{R} = \left\{ 0\right\}$ and that there exist $\beta \geq 1, \gamma > 0$ such that 
$$
\left\| (i\lambda I - H)^{-1}\right\|_{\mathcal{L}(X)} \leq \left\{
\begin{array}{ll}
O(|\lambda|^{-\beta}), \, \lambda \rightarrow 0, \\
O(|\lambda|^{\gamma}), \, \lambda \rightarrow \infty.
\end{array}
\right.
$$
Then, there exists  a constant $C > 0$ such that  
$$
\left\|\mathcal{T}(t) z \right\|_{X} \leq 
\frac{C}{{\bf t}^{\frac{1}{\max{(\beta,\gamma)}}}} \left\|z\right\|_{{\mathcal D}(H)} , \; \forall \, t > 0, z \in {\mathcal D}(H) \cap R(H),
$$
where ${\mathcal D}(H)$ is the domain of $H$ and $R(H)$ is the range of $H.$
 \end{theorem}

Based on Theorem \ref{TBC}, a simple adaptation of the proofs of \cite[Theorems 5.8 and 5.9]{abbes} (see also \cite{kais,kaisI}) leads to the following stability result: 

 \begin{proposition}
If $\eta=0$, then there exists a $C > 0$ such that 
\begin{equation}
\label{eta}
\left\|e^{{\bf t}\mathcal{A}} U_0\right\|_{\mathcal{H}}  \leq 
\frac{C}{\sqrt{{\bf t}}} \left\|U_0\right\|_{\mathcal{D}(\mathcal{A})}, \, \forall \, {\bf t} >  0, \, U_0 \in {\mathcal D}(\mathcal{A}) \cap R(\mathcal{A}),
\end{equation}
where $R(\mathcal{A})$ is the range of $\mathcal{A}.$
 \end{proposition}

\section{Numerical Approximation}

In this section, we will verify numerically the polynomial rate of decay obtained in the previous
section. For simplicity, we consider numerical examples in dimension
$n=2$. The real case
$n=3$ is merely a general physical description that, due to symmetry, can be reduced to two dimensions
$\Omega=(0,l_1)\times(0,l_2)$.

\subsection{Finite Volume Approximation}
We consider the finite volume method (FVM) for spatial discretization of the variables 
$u=u({\pmb{\color{black}{x}}},t)$ and 
$v=v({\pmb{\color{black}{x}}},t)$
with ${\pmb{\color{black}{x}}}=(x,y)$, based on a discretization of finite differences of flux \cite{Eymard}. In this sense, 
let $\mathcal{T}=\bigcup\limits_{i,j}K_{ij}$ be a uniform discretization of the rectangle 
$\Omega=(0,l_1)\times(0,l_2)$
in small $I\times J$ control volumes 
$K_{ij}=(x_{i-\frac{1}{2}},x_{i+\frac{1}{2}})\times
(y_{j-\frac{1}{2}},y_{j+\frac{1}{2}})$,
with $x_{i+\frac{1}{2}}=i\delta x$, 
$y_{j+\frac{1}{2}}=j\delta y$,
$\delta x = \dfrac{l_1}{I}$, $i=0,\ldots,I$
$\delta y = \dfrac{l_2}{J}$, $j=0,\ldots,J$.
The unknowns $u(x,y,t)$ and $v(x,y,t)$, are approximated by $\mathbf{u}=u_{ij}(t)$ and 
$\mathbf{v}=v_{ij}(t)$ respectively in the control volume $K_{ij}$. 
Given the uniformity of the mesh, the diffusion term is approximated by
$$
\Delta u(x_i,y_j) \approx 
\left(\mathbf{D}^2 \mathbf{u}\right)_{ij}
=
\dfrac{u_{i+1,j}-2u_{i,j}
+u_{i-1,j}}{\delta x^2}
+
\dfrac{u_{i,j+1}-2u_{i,j}
+u_{i,j-1}}{\delta y^2}
$$
with $u_{0,j}=u_{i,0}=u_{i,J+1}=u_{I+1,j}=
u_{0,0}=u_{I+1,0}=u_{0,J+1}=u_{I+1,J+1}=0$,
$i=1,\ldots,I$, $j=1,\ldots,J$, 

Integrating the system \eqref{105}
into the control volume $K_{ij}$
and taking into account the previous approximation of the Laplacian operator,
we deduce the following numerical scheme
\begin{equation}
\begin{cases}
    \rho_1 \Ddot{\mathbf{u}} - a_{11} \mathbf{D}^2 \mathbf{u} - a_{12} \mathbf{D}^2 \mathbf{v}
    +\alpha (\mathbf{u}-\mathbf{v}) 
    + \mathbf{D}^{\alpha,\eta} \mathbf u =0, \quad t > 0,\\
     \rho_2 \Ddot{\mathbf{u}} - a_{21} \mathbf{D}^2 \mathbf{u} - a_{22} \mathbf{D}^2 \mathbf{v}
    -\alpha (\mathbf{u}-\mathbf{v}) 
    + \mathbf{D}^{\alpha,\eta} \mathbf v =0,\quad t > 0,\\
    \mathbf{D}^{\alpha,\eta}{\mathbf{u}}(t) =  \mathfrak{C}\int_{\mathbb{R}}\mu(\xi)
{\pmb{\color{black}{\phi}}}(t,\,\xi)d\xi,\quad t > 0,\\
\mathbf{D}^{\alpha,\eta}{\mathbf{v}}(t) =  \mathfrak{C}\int_{\mathbb{R}}\mu(\xi)
{\pmb{\color{black}{\psi}}}(t,\,\xi)d\xi,\quad t > 0,\\
\pmb{\color{black}{\phi}}_{t}(t,\,\xi) + |\xi|^{2}\pmb{\color{black}{\phi}}(t,\,\xi) =
\mu(\xi)\mathbf{u}(t),\quad \xi\in\mathbb{R},\quad t > 0,\\
\pmb{\color{black}{\psi}}_{t}(t,\,\xi) + |\xi|^{2}\pmb{\color{black}{\psi}}(t,\,\xi) =
\mu(\xi)\mathbf{v}(t),\quad \xi\in\mathbb{R},\quad t > 0,
\end{cases}
\label{scheme}
\end{equation}
where 
$\mathbf{\Phi}=[{\pmb{\color{black}{\phi}}}(\xi,t),{\pmb{\color{black}{\psi}}}(\xi,t)]^T$ is
an 
approximation of the
solutions of \eqref{209}$_3$ and \eqref{209}$_4$.

\subsection{Linear equations of Motion}
    Let the vector 
$\mathbf{U}=[\mathbf{u}(t), \mathbf{v}(t)]^\top$ is
an 
approximation of $[u,v]^\top$
in $\mathbb{R}^{2IJ}$.
Considering \eqref{scheme}, we have the following system of equations of motion
\begin{equation}
	\label{LEM}
	\mathbf{M}\ddot{\mathbf{U}}(t)+
	\mathbf{K}{\mathbf{U}}(t)+
	\mathbf{C}\overset{\alpha,\eta}{\mathbf{U}}(t)
	=0
\end{equation}
where $\mathbf{M}$ is the $2IJ$ densities matrix
given by:
\begin{equation}
    \label{M}
\mathbf{M}=
\begin{pmatrix}
\rho_1\mathbf{I}_{IJ\times IJ} &  \mathbf{O}_{IJ\times IJ}\\
 \mathbf{O}_{IJ\times IJ} &  \rho_2\mathbf{I}_{IJ\times IJ}
\end{pmatrix},
\end{equation}
with $\mathbf{I}_{IJ\times IJ}$ and $\mathbf{O}_{IJ\times IJ}$ the identity and null matrices respectively of $\mathbb{R}^{IJ}$; 
The stiffness matrix $\mathbf{K}$ is given by
\begin{equation}
    \label{K}
\mathbf{K}=
\begin{pmatrix}
-a_{11} \mathbf{D}^2 & -a_{12} \mathbf{D}^2\\
-a_{21} \mathbf{D}^2 & -a_{22} \mathbf{D}^2
\end{pmatrix},
\end{equation}
where
$\mathbf{C}\overset{\alpha,\eta}{\mathbf{U}}(t)= D^{\alpha,\eta}{\mathbf{U}}(t)$
is the generalized Caputo fractional derivative defined in \eqref{202}. 

\subsection{Time discretization}

In order to preserve the energy with a second order scheme in time, we choose a $\beta$-Newmark scheme for $w$.
The method  consists of updating
the displacement, velocity and acceleration vectors
 at the current time $t^n=n\delta t$ to the time $t^{n+1} = (n+1)\delta t$, a small time interval
 $\delta t$
later.
The Newmark algorithm \cite{Newmark}
is based on a set of two relations expressing the forward displacement
$\mathbf{U}^{n+1}$
 and velocity
$\dot{\mathbf{U}}^{n+1}$  in terms of their
current values and the forward and current values of the acceleration:
\begin{eqnarray}
\dot{\mathbf{U}}^{n+1} &=& \dot{\mathbf{U}}^{n} + (1-\widetilde{\gamma})\delta t\,\ddot{\mathbf{U}}^{n} + \widetilde{\gamma}\delta t\,\ddot{\mathbf{U}}^{n+1},\label{702}\\
{\mathbf{U}}^{n+1} &=& {\mathbf{U}}^{n} + \delta t \dot{\mathbf{U}}^{n} + \left(\frac{1}{2}-\widetilde{\beta}\right)\delta t^ 2\,
\ddot{\mathbf{U}}^{n} + \widetilde{\beta}\delta t^2\,\ddot{\mathbf{U}}^{n+1},\label{703}
\end{eqnarray}
where $\widetilde{\beta}$ and $\widetilde{\gamma}$ are parameters of the methods that will be fixed later.
Replacing \eqref{702}--\eqref{703} in the equation of motion \eqref{LEM}, we obtain
\begin{equation}
      \left(
    \mathbf{M}
    +  \widetilde{\beta}\delta t^2\,
    \mathbf{K}
\right)
\ddot{\mathbf{U}}^{n+1}
+ 
	\mathbf{C} \halfscript{\overset{\alpha,\eta}{\mathbf{U}}}{n+1}
=
 -
  \mathbf{K}
\left( \mathbf{U}^n+\delta t \dot{\mathbf{U}}^n + \left(\dfrac{1}{2}- \widetilde{\beta}\right) \delta t^2 \ddot{\mathbf{U}}^n\right).
 \label{NewmarkW}
\end{equation}

\subsection{Approximation of fractional derivatives}

Let $\xi_\ell:=\ell\delta \xi$
 $\ell=1,\ldots,M$, 
 $\delta\xi=M/R$.
From \eqref{203}, we define
\begin{eqnarray*}
\label{629}\mu_\ell = |\xi_\ell|^{(2\alpha - 1)/2},\quad
\ell=1,\ldots,M,\quad 0 < \alpha < 1.
\end{eqnarray*}
Thus, an approximation of the integral \eqref{208}, is given by
\begin{eqnarray}
\label{631}
\mathbf{C}\halfscript{\overset{\alpha,\eta}{\mathbf{U}}}{n}\approx
2\mathfrak{C}  \displaystyle\sum_{\ell=1}^M\mu_\ell \mathbf{\Phi}^{n}_\ell\delta\xi.
\end{eqnarray}
On the other hand, the system \eqref{scheme}$_5$-\eqref{scheme}$_6$ can be discretized using the Crank–Nicolson method \cite{Crank}, in order to maintain the conservation of energy, or its nondecrease in case of dissipation.
Combining it with the Newmark scheme \eqref{NewmarkW}
we obtain
the following  conservative numerical scheme:
\begin{equation}
    \begin{cases}
      \left(
    \mathbf{M}
    +  \widetilde{\beta}\delta t^2\,
    \mathbf{K}
\right)
\ddot{\mathbf{U}}^{n+1}
+ 
	2\mathfrak{C}  \displaystyle\sum_{\ell=1}^M\mu_\ell \mathbf{\Phi}^{n+1}_\ell\delta\xi
=
 -
  \mathbf{K}
\left( \mathbf{U}^n+\delta t \dot{\mathbf{U}}^n + \left(\dfrac{1}{2}- \widetilde{\beta}\right) \delta t^2 \ddot{\mathbf{U}}^n\right),\\
\mathbf{\Phi}^{n+1}_\ell = \mathbf{\Phi}^{n}_\ell
- \delta t \left(\xi_\ell^2 +\eta\right) \mathbf{\Phi}^{n+\frac{1}{2}}_\ell
+ \delta t \mu_\ell  \dot{\mathbf{U}}^{n+\frac{1}{2}},
\end{cases}
 \label{ConservNumerical}
\end{equation}
where 
$\mathbf{\Phi}^{n+\frac{1}{2}}_\ell
=\dfrac{\mathbf{\Phi}^{n}_\ell+\mathbf{\Phi}^{n+1}_\ell}{2}
$.
Using  \eqref{702} again and replacing in 
\eqref{ConservNumerical}$_2$,
we can rewrite \eqref{ConservNumerical}
in the following more explicit and computable way:
\begin{equation}
    \begin{cases}
      \left(
    \mathbf{M}
    +  \widetilde{\gamma}\delta t\,
    \mathbf{C}_{augm}
    +  \widetilde{\beta}\delta t^2\,
    \mathbf{K}
\right)
\ddot{\mathbf{U}}^{n+1}
=
-
	2\mathfrak{C}  \displaystyle\sum_{\ell=1}^M\widetilde{\mu}_\ell \mathbf{\Phi}^{n}_\ell\delta\xi
 -\mathbf{C}_{augm}\left(
2\dot{\mathbf{U}}^n + \left( 1 - \widetilde{\gamma} \right)
\delta t \ddot{\mathbf{U}}^n
\right)
  \\
\qquad\qquad\qquad\qquad\qquad\qquad\qquad - 
\mathbf{K}
\left( \mathbf{U}^n+\delta t \dot{\mathbf{U}}^n + \left(\dfrac{1}{2}- \widetilde{\beta}\right) \delta t^2 \ddot{\mathbf{U}}^n\right),
\\
\mathbf{\Phi}^{n+1}_\ell = \dfrac{2-\delta t\left(\xi_\ell^2 +\eta\right)}{2+\delta t\left(\xi_\ell^2 +\eta\right)} \mathbf{\Phi}^{n}_\ell
+ 
\dfrac{2\delta t \mu_\ell }{2+\delta t\left(\xi_\ell^2 +\eta\right)}
\dot{\mathbf{U}}^{n+\frac{1}{2}},
\end{cases}
 \label{ExplicitWay}
\end{equation}
where  $\widetilde{\mu}_\ell=\dfrac{2-\delta t\left(\xi_\ell^2 +\eta\right)}{2+\delta t\left(\xi_\ell^2 +\eta\right)}
\mu_\ell$ and $\mathbf{C}_{augm}=\delta t \mathfrak{C}  \left(\displaystyle\sum_{\ell=1}^M
\dfrac{2{\mu}_\ell^2\delta\xi}{2+\delta t\left(\xi_\ell^2 +\eta\right)}\right)\textbf{I}_{2IJ} $.

\subsection{Decay of the discrete energy}

Evaluating \eqref{ConservNumerical}$_1$ in $t=t_{n+\frac{1}{2}}$, multiplying 
by $\ddot{\mathbf{U}}^{n+\frac{1}{2}}$, and summing
\eqref{ConservNumerical}$_2$ multiplied by 
$\mathfrak{C}\mathbf{\Phi}^{n+\frac{1}{2}}_\ell$, we obtain that
\begin{eqnarray}
\left[ E_\Delta \right]_n^{n+1} &:=&
 \left[ 
 \dfrac{1}{2} \dot{\mathbf{U}}^T  \mathbf{M} \dot{\mathbf{U}}
 +
  \dfrac{1}{2} {\mathbf{U}}^T  \mathbf{K} {\mathbf{U}}
  +
 \dfrac{\mathfrak{C}}{2}  \displaystyle\sum_{\ell=1}^M {\mu}_\ell \left|\mathbf{\Phi}_\ell\right|^2
  \right]_n^{n+1}\\
  &=&\nonumber
  - \ \mathfrak{C}
   \displaystyle\sum_{\ell=1}^M  \left(\xi_\ell^2 +\eta\right) \left|\mathbf{\Phi}^{n+\frac{1}{2}}_\ell\right|^2
   \label{ener_Mbodje}
\end{eqnarray}
which is consistent with the estimates \eqref{313}--\eqref{314}, and constitutes a correct approximation of the energy and its decreasing behavior.

\subsection{Example 1: Interaction of two waves with and without the dispative fractional derivative term.}
In this first example, we interact two waves each coming from a different side of the domain for each of the compounds in the mixture. The idea is to highlight the effect of the dissipative fractional derivative term, and to do so, we compare two simulations: one with dissipation and the other without dissipation, that is, the same system but without the fractional derivative terms.

\begin{figure*}[hbt]
    \centering
    \begin{subfigure}[b]{0.5\textwidth}
        \centering
        \includegraphics[scale=0.5]{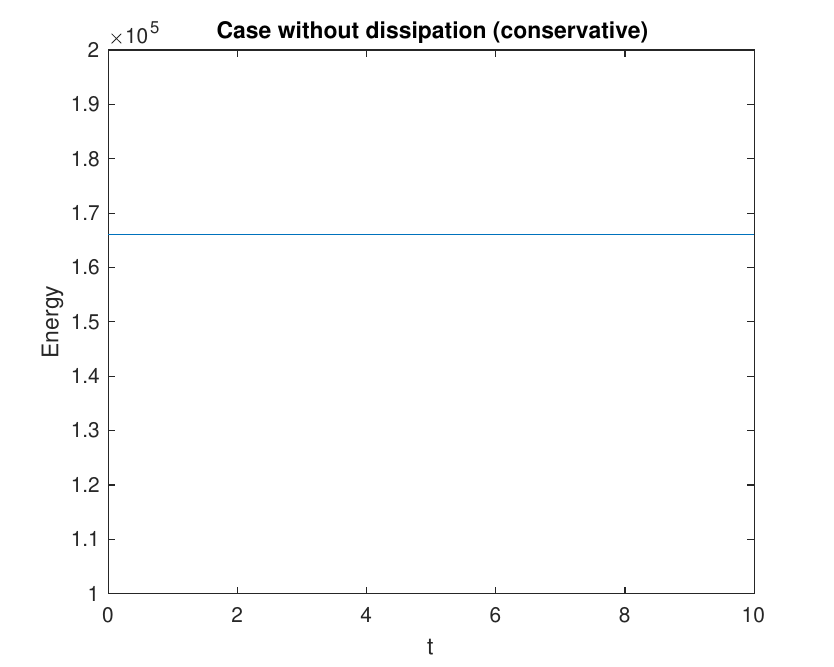}
        \caption{Energy without dissipation}
    \end{subfigure}%
    \begin{subfigure}[b]{0.5\textwidth}
        \centering
        \includegraphics[scale=0.38]{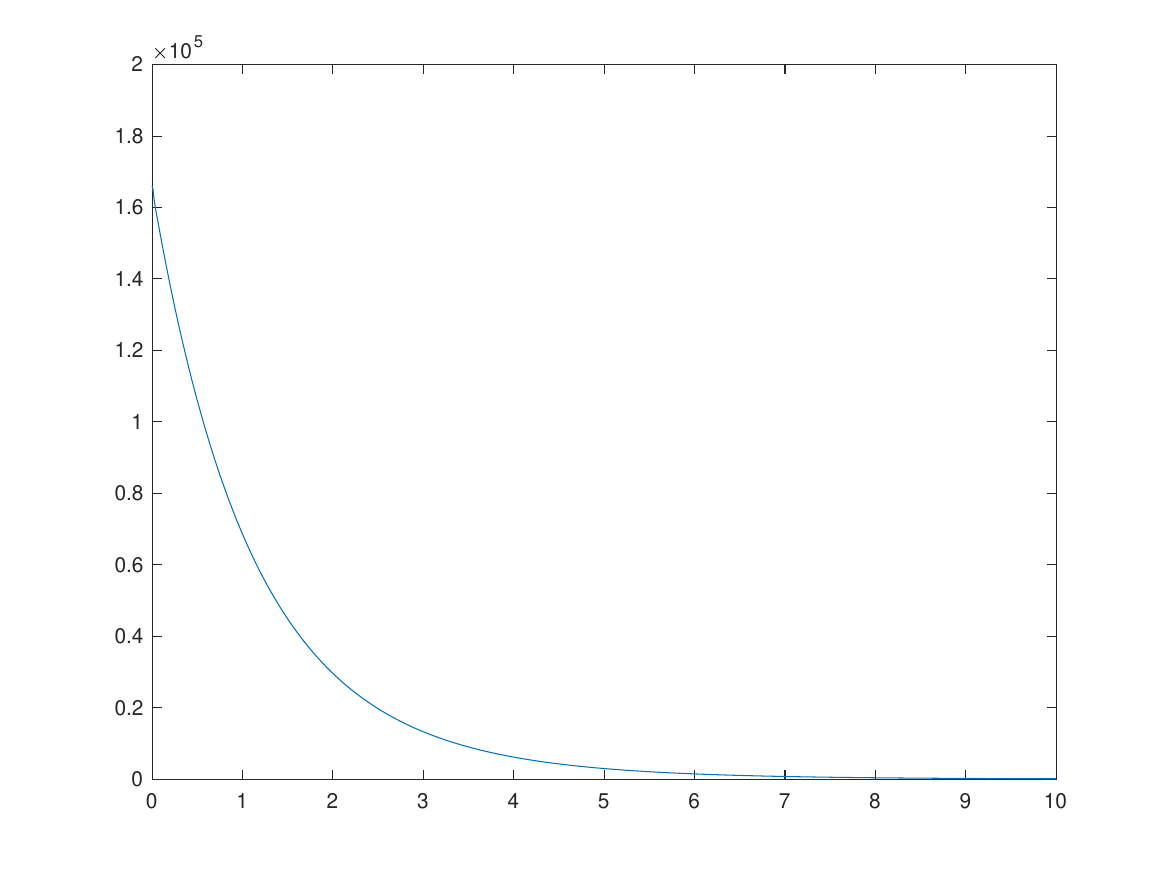}
        \caption{Energy with fractional derivative dissipation}
    \end{subfigure}    
    \caption{Numerical comparison of energy with and without dissipative term of fractional derivative.}
\end{figure*}

We consider the numerical scheme
of the mixture system in the square $\Omega=(0,1)^2$ with the initial 
condition
\begin{align*}
 &   u_{0}({\pmb{\color{black}{x}}})
=
\dfrac{1}{\sigma\sqrt{2}}e^{-\dfrac{(x-\frac{1}{2})^2+y^2}{2\sigma^2}},
&
  u_{1}({\pmb{\color{black}{x}}})
=0,\\
 &   v_{0}({\pmb{\color{black}{x}}})
=
\dfrac{1}{\sigma\sqrt{2}}e^{-\dfrac{x^2+(y-\frac{1}{2})^2}{2\sigma^2}},
&
  v_{1}({\pmb{\color{black}{x}}})
=0,
\end{align*}
with $\sigma=0.01$,
which represent two initial Gaussians, one on the $x-$axis for the compound $u({\pmb{\color{black}{x}}})$, and the other on the $y-$axis, for the compound $v({\pmb{\color{black}{x}}})$.

The system parameters are in this case, densities $\rho_1=\rho_2=1$, a matrix $A$ of coefficients $a_{11}=0.1$,
$a_{12}=-0.5$, $a_{21}=-0.5$ and $a_{22}=0.1$ (which is positive symmetric), and finally a coupling coefficient $\alpha=1.0$.
We run a simulation with $I=J=200$, $M=1000$, and $\delta\xi=0.1$. On the other hand, we simulate up to a final time $T=10$ and $\delta t=0.01$.

In Figure 1, the numerical behavior of the energy can be observed: the graph on the left shows the energy without dispassion terms, which remains constant as predicted by the theory and the numerical scheme \eqref{ener_Mbodje}, with its right side zero instead of the term that multiplies $\mathfrak{C}$; the graph on the right, on the other hand, shows a decreasing behavior of the energy.

\begin{figure*}[hbt]
    \centering
    \begin{subfigure}[b]{0.5\textwidth}
        \centering
        \includegraphics[scale=0.5]{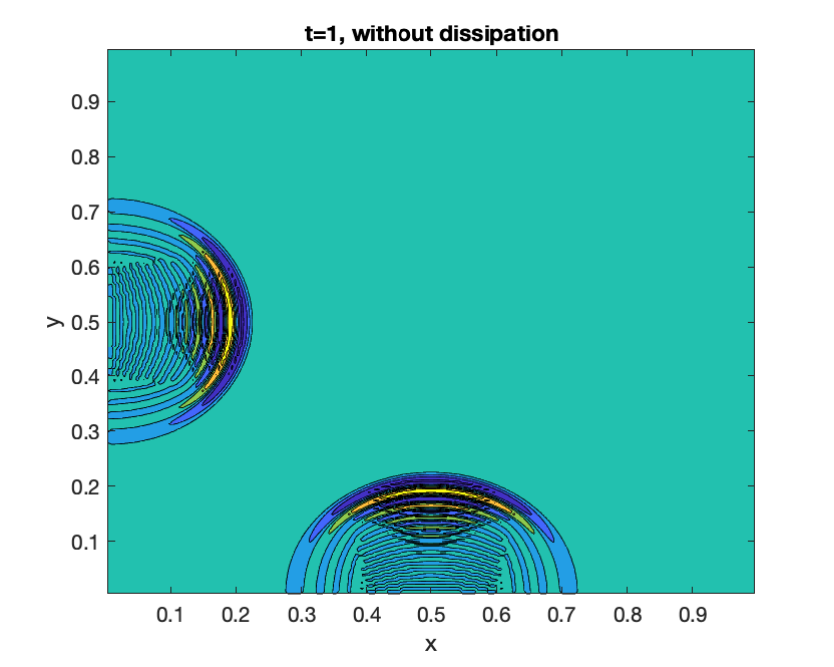}
        \caption{$\mathbf{u}+\mathbf{v}$ without dissipation, $t=1$}
    \end{subfigure}%
    \begin{subfigure}[b]{0.5\textwidth}
        \centering
        \includegraphics[scale=0.5]{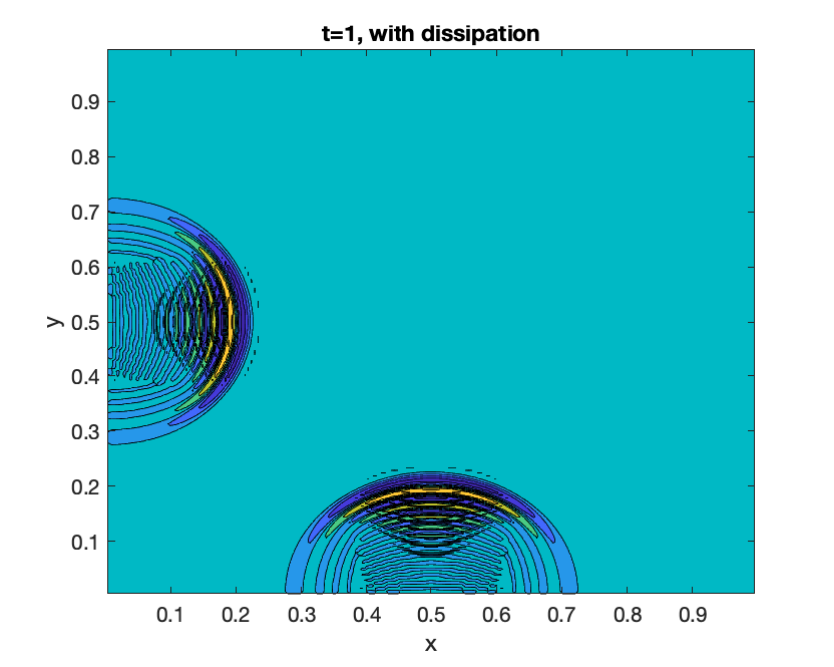}
        \caption{$\mathbf{u}+\mathbf{v}$ with dissipation, $t=1$}
    \end{subfigure}
    \\
        \begin{subfigure}[b]{0.5\textwidth}
        \centering
        \includegraphics[scale=0.5]{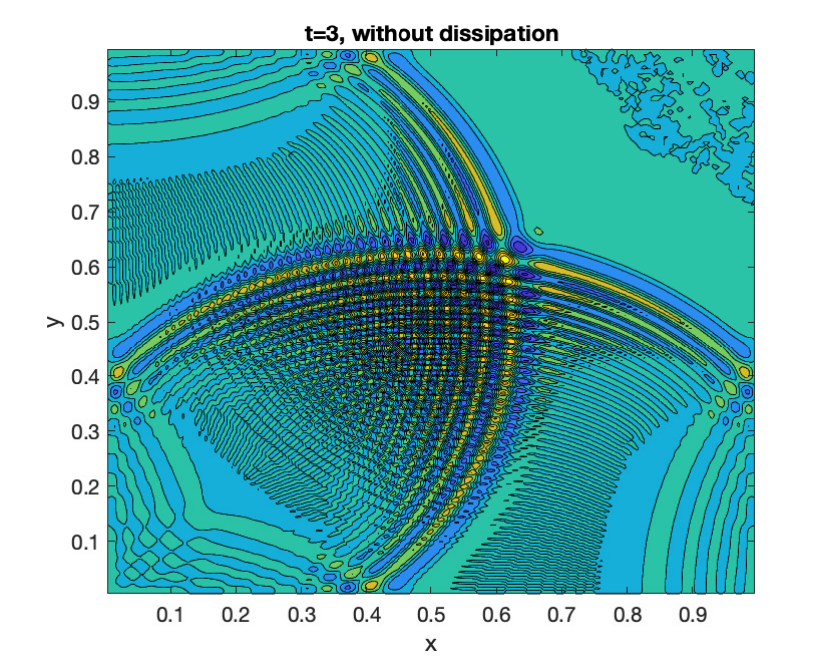}
        \caption{$\mathbf{u}+\mathbf{v}$ without dissipation, $t=3$}
    \end{subfigure}%
    \begin{subfigure}[b]{0.5\textwidth}
        \centering
        \includegraphics[scale=0.5]{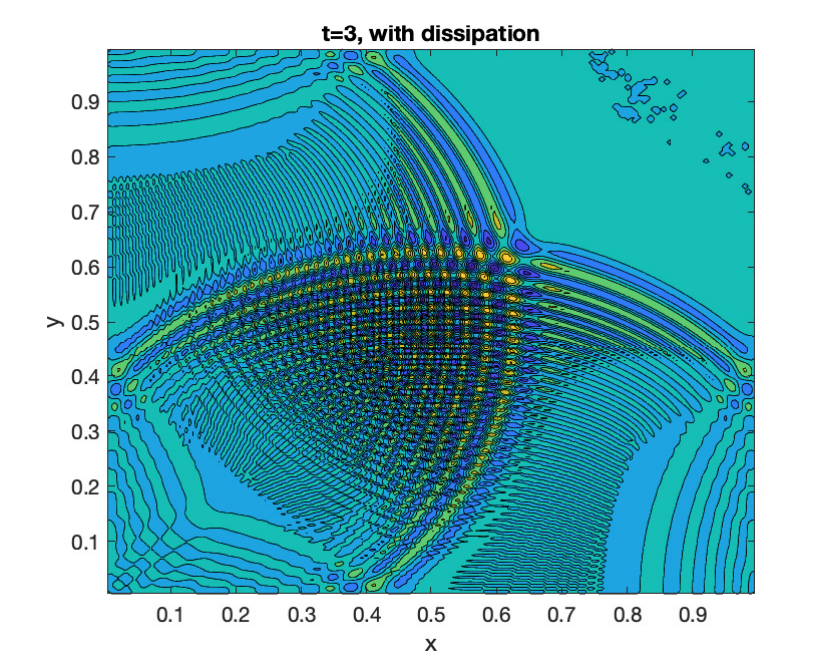}
        \caption{$\mathbf{u}+\mathbf{v}$ with dissipation, $t=3$}
    \end{subfigure}
\\
    \begin{subfigure}[b]{0.5\textwidth}
        \centering
        \includegraphics[scale=0.5]{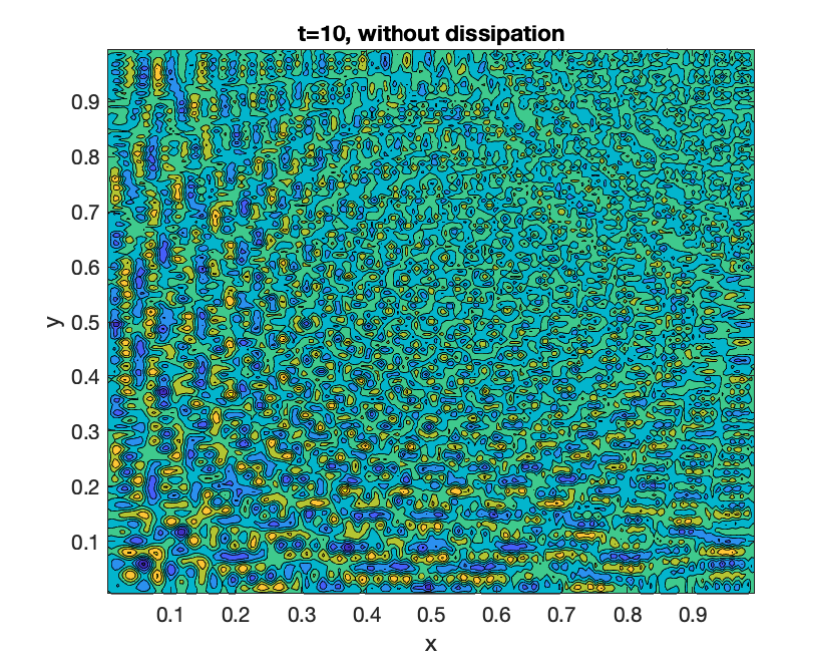}
        \caption{$\mathbf{u}+\mathbf{v}$ without dissipation, $t=10$}
    \end{subfigure}%
    \begin{subfigure}[b]{0.5\textwidth}
        \centering
        \includegraphics[scale=0.5]{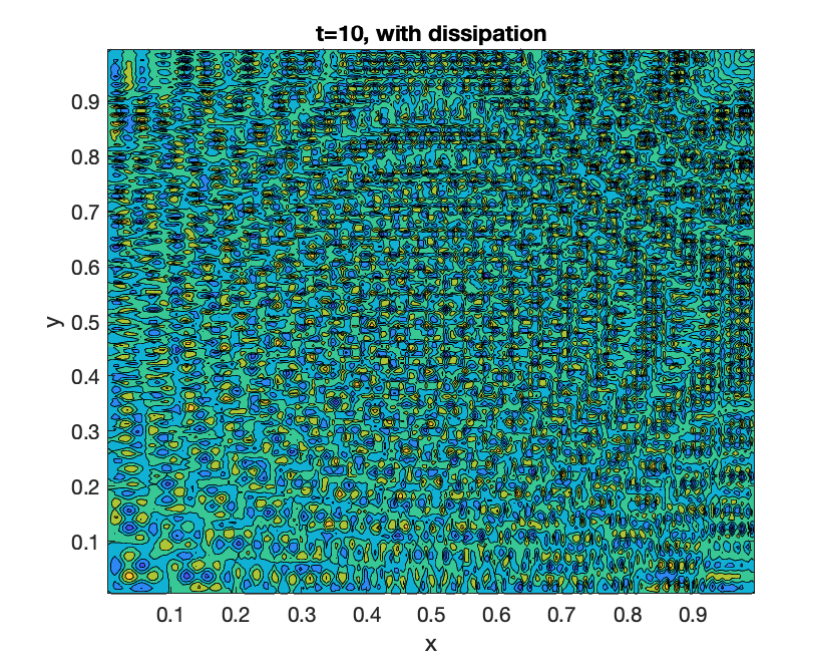}
        \caption{$\mathbf{u}+\mathbf{v}$ with dissipation, $t=10$}
    \end{subfigure}
    \caption{Behavior of the sum of both compounds u+v (mixture), for different times without dissipation (graphs on the left), and with dissipation (graphs on the right).}
\end{figure*}

In Figure 2, the interaction of the waves can be observed both with and without dispersion. From a qualitative perspective, the behaviors are very similar (the differences between the figures on the left and right are subtle to the naked eye), with only the behavior of the energy marking the difference between the two.

\subsection{Example 2: Numerical
spectrum of operator $A$.}
In this subsection, the aim is to visualize the spectrum of the operator 
$A$, defined in \eqref{304} and approximated by the numerical scheme \eqref{ExplicitWay}, that is, by the conservative finite difference matrix derived from the augmented model. Considering the conservative scheme \eqref{ConservNumerical}$_2$ for approximating the variables 
$\varphi_1$ and $\varphi_2$, an 
$NxN$ matrix is obtained, where 
$N=2(M+2)IJ$, defined by:$$
\textbf{A}
=
\begin{bmatrix}
    \mathbf{O} & \mathbf{M} & \mathbf{O} & \mathbf{O} & \cdots & \mathbf{O} \\
    -\mathbf{K} & \mathbf{O} &
    -\mathfrak{C}\delta\xi\mu_1\mathbf{I} & 
    -\mathfrak{C}\delta\xi\mu_2\mathbf{I} & 
    \cdots & -\mathfrak{C}\delta\xi\mu_M\mathbf{I} &   \\
    \mathbf{O} & \mu_1\mathbf{I} 
    & -(\xi_1^2+\eta)\mathbf{I} 
    & \mathbf{O} & \cdots & \mathbf{O}\\
    \mathbf{O} & \mu_2\mathbf{I} &
    \mathbf{O} & -(\xi_2^2+\eta)\mathbf{I} 
    &\ddots & \mathbf{O}
    \\
    \vdots & \ddots & \ddots & \ddots & \ddots &\vdots \\
    \mathbf{O} &\mu_M\mathbf{I} &
    \mathbf{O} & \cdots &\cdots & 
    -(\xi_M^2+\eta)\mathbf{I} 
\end{bmatrix}
$$
where $\mathbf{M}$, $\mathbf{K}$ are defined in \eqref{M}-\eqref{K}, and 
$\mathbf{O}$,  $\mathbf{I}$
are the null and identity matrices, respectively, of $(2IJ)\times(2IJ)$.
Since we want to visualize the complete set of eigenvalues, we will perform a visualization only for $M=1$, $2$, and $3$. On the other hand, by taking values of $I=J=100$, we will obtain matrices large enough to be treated numerically of $N\times N$, with $N=120,000$, $N=160,000$ and $N=200,000$, for each of the 3 respective values of $M$.
\begin{figure*}[hbt]
    \centering
    \begin{subfigure}[b]{0.5\textwidth}
        \centering
        \includegraphics[scale=0.5]{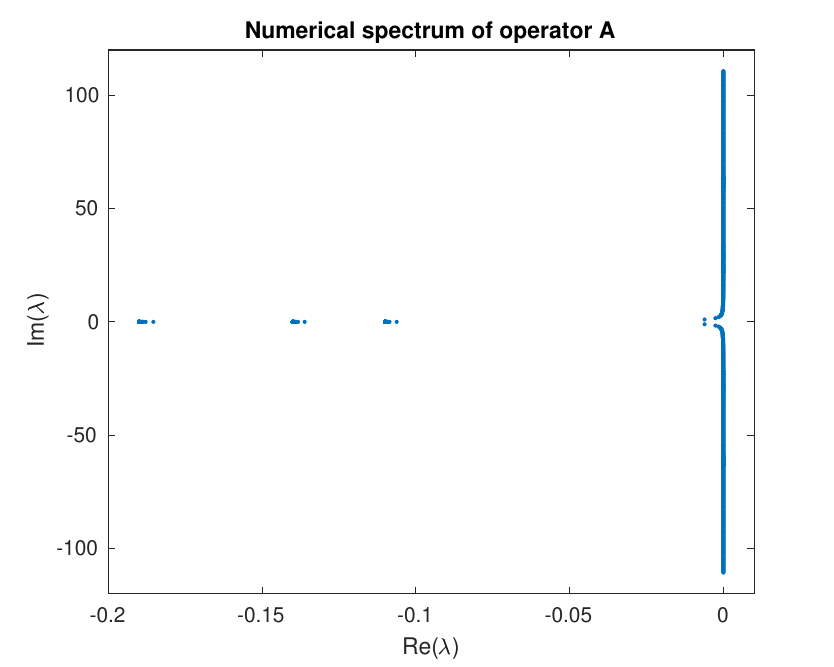}
        \caption{Numerical Spectrum with $M=3$}
    \end{subfigure}%
    \begin{subfigure}[b]{0.5\textwidth}
        \centering
        \includegraphics[scale=0.5]{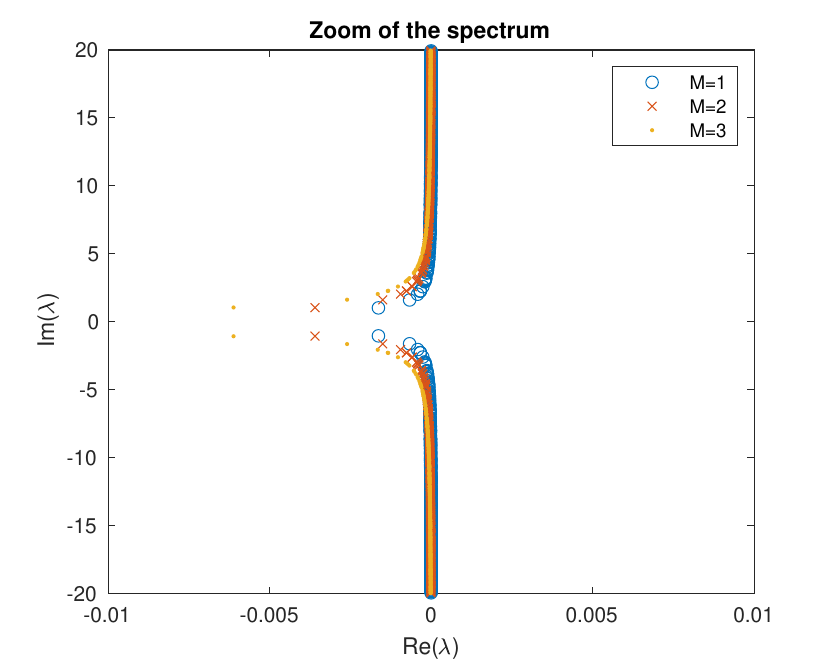}
        \caption{Zoom of spectrum for different value of $M$}
    \end{subfigure}    
    \caption{Numerical spectrum of the operator $A$, using FVM approximation.}
\end{figure*}
\\
\begin{table}[hbt]
    \centering
    \begin{tabular}{c|c}
        $M$ & Eigenvalues of largest module \\
        \hline\hline
       1  & $-1.4303\cdot 10^{-07}  \pm  110.64i$ \\
       \hline
       2 & $-3.2507\cdot 10^{-07}  \pm  110.64i$ \\
       \hline
       3 & $-5.7212\cdot 10^{-07}  \pm  110.64 i$\\
       \hline
       \end{tabular}
    \caption{Eigenvalues of the largest module, and closest to the imaginary axis for different values of $M$.}
    \label{tab:my_label}
\end{table}
In the left panel of Figure 3, the full spectrum of matrix 
$A$ is shown for 
$M=3$. All eigenvalues lie entirely on the left half of the complex plane and are close to the imaginary axis, which confirms the stability of the numerical method, as well as its non-exponential decay. In fact, in this case, the eigenvalues with the largest modulus are the closest to the imaginary axis (see Table 1). It is also observed that, as the mesh is refined, that is, as
$I$ and $J$ increase, eigenvalues with larger modulus (corresponding to higher frequencies) appear and become increasingly closer to the imaginary axis. The right panel of Figure 3 presents a zoom in of the spectrum near the origin, which reveals slight differences for different values of $M$.
\subsection{Example 3: Numerical
Polynomial decay of the Energy}

In this last example, we demonstrate the polynomial decay behavior of the energy, simulating for long times. Furthermore, it is clear from the previous example (spectrum visualization) that solutions that capture high frequencies are important for visualizing polynomial and non-exponential decay, since high frequencies correspond to the eigenvalues closest to the imaginary axis. To do this, we consider a non-regular initial condition that still belongs to the $\mathcal{D}(A)$ domain. In this case, we choose the following initial condition:
\begin{align*}
 &   u_{0}({\pmb{\color{black}{x}}})
=
\begin{cases}
    1-10\sqrt{(x-\frac{1}{2})^2+(y-\frac{1}{2})^2},
    & \text{if }
    \|(x-\frac{1}{2},y-\frac{1}{2})\|<\frac{1}{10}\\
    0 & \text{otherwise,}
\end{cases}
&
  u_{1}({\pmb{\color{black}{x}}})
=0,
\end{align*}
and
\begin{align*}
 &   v_{0}({\pmb{\color{black}{x}}})
=
0,
&
  v_{1}({\pmb{\color{black}{x}}})
=\begin{cases}
    \left(1-100(x-\frac{1}{2})^2-100(y-\frac{1}{2})^2\right)^3,
    & \text{if }
    \|(x-\frac{1}{2},y-\frac{1}{2})\|<\frac{1}{10}\\
    0 & \text{otherwise.}
\end{cases}
\end{align*}
It is straightforward to verify that $u_0\in H^1_0(\Omega)$ and $v_1\in H^2(\Omega)\cap H^1_0(\Omega)$, so the initial data belong to $\mathcal{D}(A)$, while not being excessively regular.
Panels (A) and (B) of Figure 4 display the shapes of the initial conditions, while panels (C) and (D) illustrate the rapid decay of the solutions over long times, highlighting the (in this case, polynomial) decay of the energy.
\begin{figure*}[hbt]
    \centering
    \begin{subfigure}[b]{0.5\textwidth}
        \centering
        \includegraphics[scale=0.5]{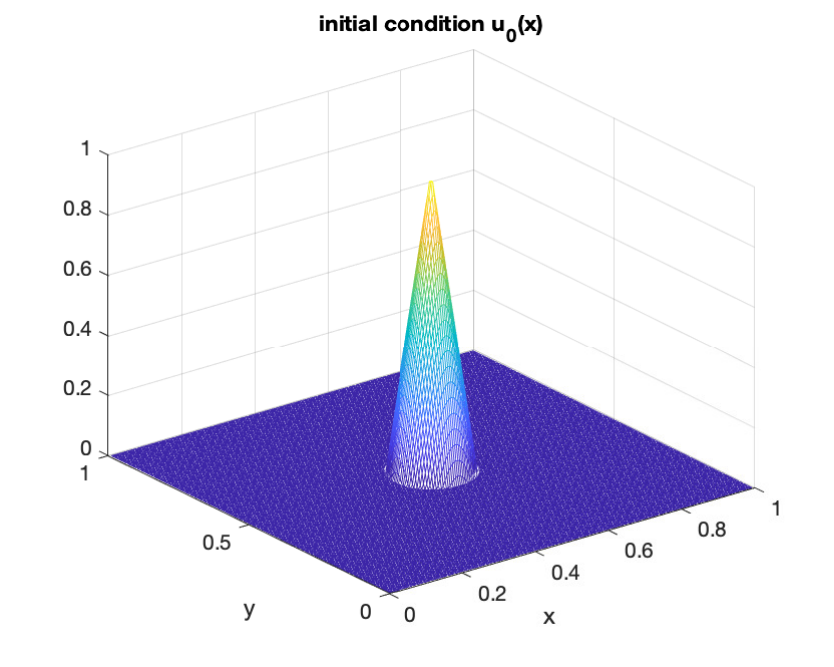}
        \caption{Conical initial condition $u_0(x,y)$}
    \end{subfigure}%
    \begin{subfigure}[b]{0.5\textwidth}
        \centering
        \includegraphics[scale=0.5]{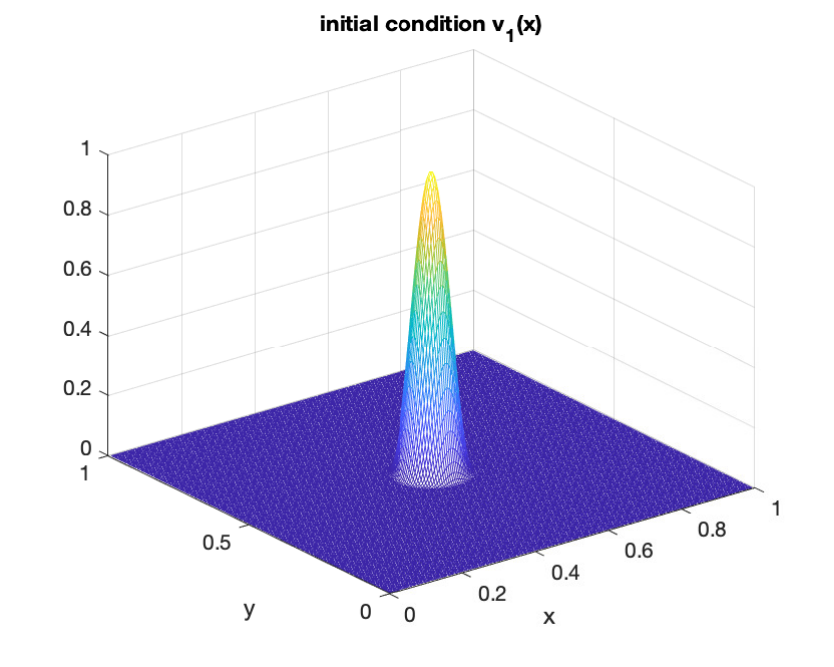}
        \caption{Polynomial of degree 6, initial condition $v_1(x,y)$}
    \end{subfigure}    
    \\
  \centering
    \begin{subfigure}[b]{0.5\textwidth}
        \centering
        \includegraphics[scale=0.5]{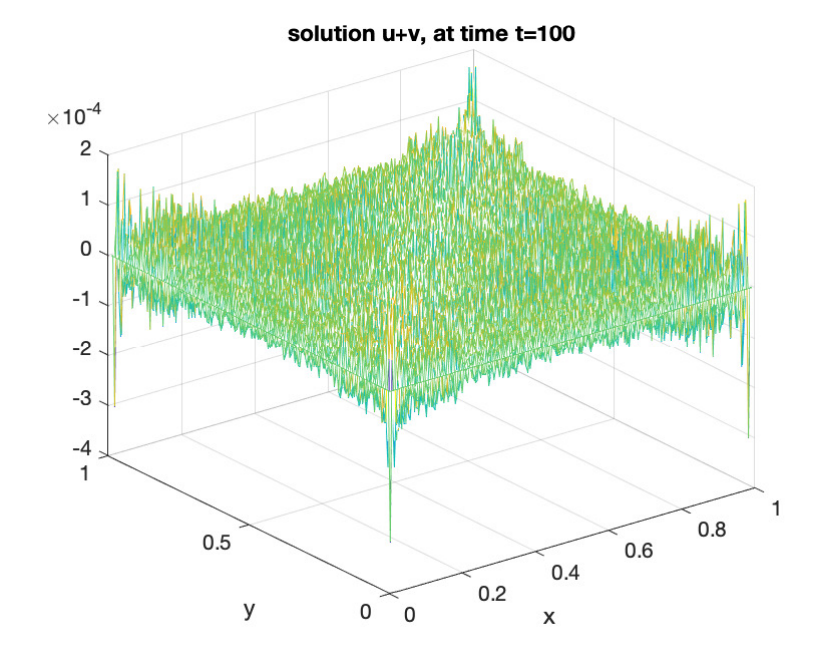}
        \caption{Solution $u+v$ at time $t=100$}
    \end{subfigure}%
    \begin{subfigure}[b]{0.5\textwidth}
        \centering
        \includegraphics[scale=0.5]{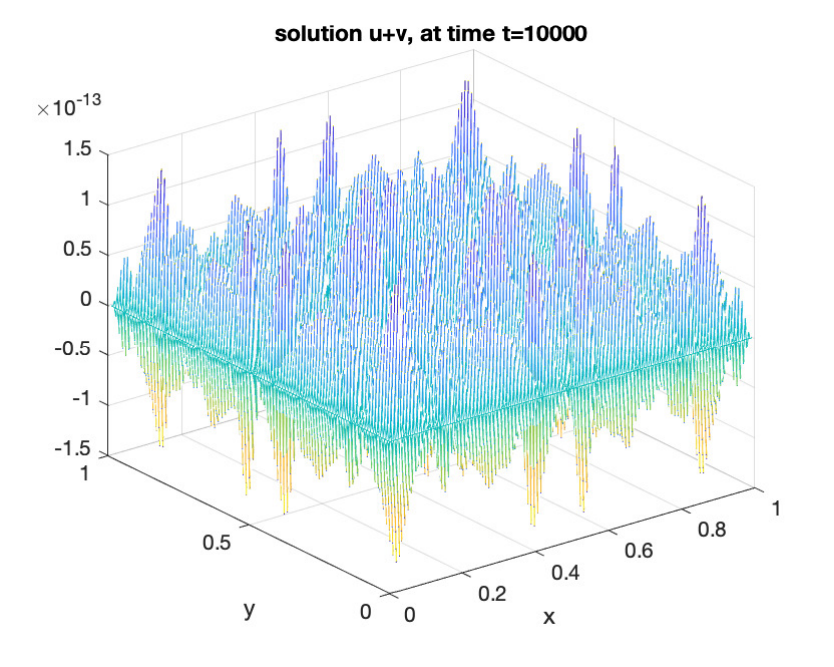}
        \caption{Solution $u+v$ at time $t=10000$}
    \end{subfigure}    
    \caption{Evolution of the solution for long times with not very smooth initial conditions
    in $\mathcal{D}(A)$.}    
\end{figure*}
Finally, Figure 5 presents a log-log plot of the energy decay, which closely resembles a polynomial decay, as indicated by the near-linear behavior typical in such plots. For a fractional derivative with 
$\alpha=0.5$, 
we compare the optimal decay rates
$1/t$ in the case 
$\eta=0$, and 
$1/t^2$  when 
$\eta>0$. 
All curves exhibit similar trends and appear closer to the $1/t$
 behavior, regardless of whether 
$\eta=0$ or not.
\begin{figure*}[hbt]
    \centering
        \includegraphics[scale=0.7]{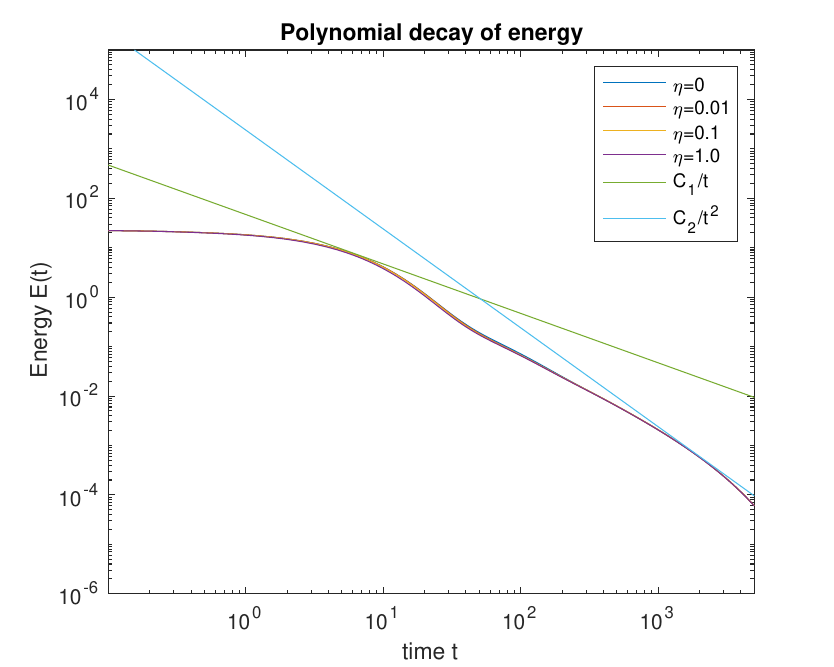}
        \caption{Energy polynomial decay in scale log-log
        and comparison with curves $\frac{C_1}{t}$ and $\frac{C_2}{t^2}$}
 \end{figure*}

\section*{Conflict of interest} 
There is no potential conflict of interest.

\end{document}